\documentclass[11pt]{amsart}

\usepackage[utf8]{inputenc}
\usepackage[T1]{fontenc}
\RequirePackage{etex} %
\usepackage[english]{babel}
\usepackage{booktabs}
\usepackage{graphicx}
\usepackage{sidecap}
\usepackage{hyperref}
\hypersetup{
	colorlinks,
	linkcolor={red!50!black},
	citecolor={blue!50!black},
	urlcolor={blue!80!black}
}
\usepackage{amssymb, amsmath,amsthm}
\usepackage{soul}
\usepackage{mathtools}
\usepackage{accents}
\usepackage{dsfont}
\usepackage{diagbox,multirow}
\usepackage{hhline}
\usepackage[dvipsnames,x11names]{xcolor}
\usepackage{cancel}
\usepackage{array}
\usepackage{colortbl}
\usepackage{subfig}
\usepackage[section]{placeins}
\usepackage{enumerate}
\usepackage[shortlabels]{enumitem}
\usepackage{makecell}
\usepackage{float}
\usepackage{framed}
\usepackage{varwidth}
\usepackage[export]{adjustbox}
\usepackage{xparse}
\usepackage {tikz}
\usetikzlibrary{cd}
\usetikzlibrary{arrows,decorations.pathmorphing,backgrounds,positioning,fit,matrix}
\usetikzlibrary{decorations.markings}
\usepackage{tikz-network}
\usepackage{todonotes}
\usepackage{pgfplots}

\newlength\mylen
\tikzset{
	bicolor/.style 2 args={
		dashed,dash pattern=on 20pt off 20pt,-,#1,
		postaction={draw,dashed,dash pattern=on 20pt off 20pt,-,#2,dash phase=20pt}
	},
}

\definecolor{palecerulean}{RGB}{137, 193, 206}

\usepackage{etoolbox}
\apptocmd{\sloppy}{\hbadness 10000\relax}{}{}
\usepackage{amsbsy}

\newtheorem{theorem}{Theorem}[section]
\newtheorem{definition}[theorem]{Definition}
\newtheorem{proposition}[theorem]{Proposition}
\newtheorem{lemma}[theorem]{Lemma}
\newtheorem{remark}[theorem]{Remark}
\newtheorem*{remark*}{Remark}
\newtheorem{corollary}[theorem]{Corollary}
\numberwithin{equation}{section}
\numberwithin{theorem}{section}

\numberwithin{equation}{section}
\numberwithin{theorem}{section}

\theoremstyle{remark}
\newtheorem{example}[theorem]{Example}
\newtheorem{cexample}[theorem]{Counterexample}

\newcommand{\X}{{\cal X}}
\newcommand{\Y}{{\cal Y}}

\newcommand{\N}{\mathbb{N}}
\newcommand{\R}{\mathbb{R}}

\renewcommand{\v}[1]{\ensuremath{#1}}

\colorlet{red2}{red!65!black}

\DeclareMathSymbol{\shortminus}{\mathbin}{AMSa}{"39}

\newcommand{\vertiii}[1]{{\left\vert\kern-0.25ex\left\vert\kern-0.25ex\left\vert #1
		\right\vert\kern-0.25ex\right\vert\kern-0.25ex\right\vert}}

\makeatletter
\newcommand{\sbullet}{%
	\hbox{\fontfamily{lmr}\fontsize{.6\dimexpr(\f@size pt)}{0}\selectfont\textbullet}}

\makeatother

\makeatletter
\providecommand\phantomcaption{\caption@refstepcounter\@captype}
\makeatother

\usepackage[style=numeric-comp,useprefix,hyperref,backend=biber,giveninits=true,maxbibnames=99]{biblatex} 
\usepackage{csquotes}
\usepackage{matlab}

\title[Non-local random walks]{Compatibility, embedding and regularization of non-local random walks on graphs}

\author{Davide Bianchi}
\address{School of Science\\
	Harbin Institute of Technology (Shenzhen)\\
	HIT Campus of University Town of Shenzhen\\
	518055 Shenzhen, CHINA}
\email{bianchi@hit.edu.cn}

\author{Marco Donatelli}
\address{Dipartimento di Scienze e Alta Tecnologia\\
	Universit\`a degli Studi dell'Insubria\\
	via Valleggio 11\\
	22100 Como, ITALY}
\email{marco.donatelli@uninsubria.it}

\author{Fabio Durastante}
\address{Dipartimento di Matematica\\
         Università di Pisa\\
         Largo Bruno Pontecorvo 5\\ 
         56127 Pisa, ITALY\\
         \& Istituto per le Applicazioni del Calcolo ``M. Picone''\\
	    Consiglio Nazionale delle Ricerche\\
	    Via Pietro Castellino 111\\
	    80131 Naples, ITALY}
\email{fabio.durastante@unipi.it}

\author{Mariarosa Mazza}
\address{Dipartimento di Scienze Umane e dell'Innovazione per il Territorio\\
	Universit\`a degli Studi dell'Insubria\\
	via Valleggio 11\\
	22100 Como, ITALY}
\email{mariarosa.mazza@uninsubria.it}

\begin{document}

\begin{abstract}
Several variants of the graph Laplacian have been introduced to model non-local diffusion processes, which allow a random walker to {\textquotedblleft jump\textquotedblright} to non-neighborhood nodes, most notably the transformed path graph Laplacians and the fractional graph Laplacian. From a rigorous point of view, this new dynamics is made possible by having replaced the original graph $G$ with a weighted complete graph $G'$ on the same node-set, that depends on $G$ and wherein the presence of new edges allows a direct passage between nodes that were not neighbors in $G$. 

We show that, in general, the graph $G'$ is not compatible with the dynamics characterizing the original model graph $G$: the random walks on $G'$ subjected to move on the edges of $G$ are not stochastically equivalent, in the wide sense, to the random walks on $G$. From a purely analytical point of view, the incompatibility of $G'$ with $G$ means that the normalized graph $\hat{G}$ can not be embedded into the normalized graph $\hat{G}'$. Eventually, we provide a regularization method to guarantee such compatibility and preserving at the same time all the nice properties granted by $G'$.    
\end{abstract}

\maketitle %



\section{Introduction}
Building non-local dynamics on graphs have been the object of deep study in many papers during the last years, allowing the random walkers to perform {\textquotedblleft long-range\textquotedblright} jumps and  to move on non-neighbors nodes. We refer to \cite{Laskin2006,Estrada,RM12,RM14,Zhao2014,Michelitsch2017,EDHMMRS} as main references and without the claim of completeness, since it is difficult to take records of all the papers and developments which arose from such original works and different fields of science. The anomalous diffusion can provide a more accurate description of the dynamics occurring in many phenomena, such as, but not just restrained to, interactions in social network and human environments (\cite{Wu2013,Baronchelli2013,Zhao2015}), diffusion processes of particles on surfaces (\cite{Wrigley1990,Senft1995,Schunack2002,Ala-Nissila2002}), porous media and disordered systems (\cite{Bouchaud1990}), etc. On the other hand, adding long-range transitions between nodes improves features like the average hitting time or the search-ability/navigability of a network. 

Starting from a given graph $G$, the main idea is to admit the passage of the random walker to nodes which are not neighbors, in a one-step move with a transition probability that depends on the combinatorial distance. There are several ways to achieve that. In 2012 two different groups independently introduced a new model to describe long-range jumps for a random walker on finite and simple graphs, \cite{Estrada,RM12}. Their approach was basically the same but from two different point of views, an algebraic one and a pure probabilistic one, respectively. The common ground is that a random walker can jump from a node to a non-neighbor node with a probability that is manually provided and that decays as a power-law of the combinatorial graph distance separating the two nodes. Rigorously speaking, this kind of dynamics is made possible by having replaced the original graph $G$ with a weighted complete graph $G'$, that depends on $G$ and wherein the presence of new edges allows a direct passage between nodes that were not neighbors in $G$, see Figure~\ref{fig:overgraph}. We will refer to those operators, and to all their subsequent modifications and generalizations that appeared later on, as \emph{path} graph Laplacians, see Subsection \ref{ssection:path_Laplacian}.

Another way to approach the same problem, that is, to introduce the possibility for a random walker to make long-range jumps, is to consider a fractional power $\alpha \in (0,1)$ of the graph Laplacian $\Delta$ associated to the original graph $G$, see Subsection \ref{ssec:fractional_graph_laplacian}. The new dynamics is obtained again thanks to the creation of a new complete graph $G'$ induced by the operator~$\Delta^\alpha$, and as an example we refer one more time to Figure~\ref{fig:overgraph}. At the best of our knowledge, the first authors to study the \emph{fractional} graph Laplacian on finite networks were A. P. Riasco and J. L. Mateos in \cite{RM14}. Even in this case, the fractional graph Laplacian is characterized by having transition probabilities that decay to zero as an $\alpha$-power of the combinatorial graph distance.

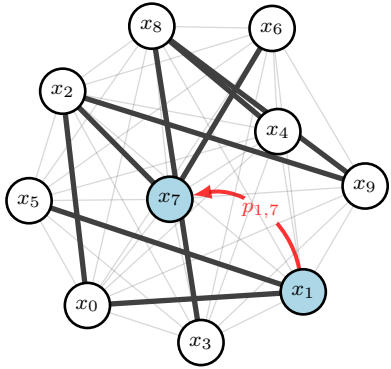
\begin{figure}[htbp]
    \centering
    \begin{tikzpicture}
        \Vertex[x=1387/938,y=-3421/2306,label=x_1,Math]{A}
        \Vertex[x=-1221/878,y=-1017/611,label=x_0,Math,color=white]{B}   
        \Vertex[x=-98/57,y=881/743,label=x_2,Math,color=white]{C}
        \Vertex[x=313/2627,y=-1487/689,label=x_3,Math,color=white]{D} 
        \Vertex[x=1891/1657,y=1141/1760,label=x_4,Math,color=white]{E}
        \Vertex[x=-1666/769,y=-903/3328,label=x_5,Math,color=white]{F}
        \Vertex[x=608/571,y=1875/928,label=x_6,Math,color=white]{G}  
        \Vertex[x=-521/1769,y=-1871/7432,label=x_7,Math]{H}  
        \Vertex[x=-694/1299,y=527/257,label=x_8,Math,color=white]{I} 
        \Vertex[x=1031/448,y=-311/4103,label=x_{9},Math,color=white]{L}
        \Edge[lw=2](A)(B)
        \Edge[lw=2](B)(C)
        \Edge[lw=2](C)(H)
        \Edge[lw=2](C)(L)
        \Edge[lw=2](G)(H)
        \Edge[lw=2](E)(I)
        \Edge[lw=2](I)(L)
        \Edge[lw=2](F)(A)
        \Edge[lw=2](D)(I)
        \Edge[lw=0.5,opacity=0.2](A)(C)
        \Edge[lw=0.5,opacity=0.2](A)(D)
        \Edge[lw=0.5,opacity=0.2](A)(E)
        \Edge[lw=0.5,opacity=0.2](A)(F)
        \Edge[lw=0.5,opacity=0.2](A)(G)
        \Edge[lw=1.5,opacity=0.8,color=red,bend=-45,label=p_{1,7},Math, Direct](A)(H)
        \Edge[lw=0.5,opacity=0.2](A)(I)
        \Edge[lw=0.5,opacity=0.2](A)(L)
        \Edge[lw=0.5,opacity=0.2](B)(D)
        \Edge[lw=0.5,opacity=0.2](B)(E)
        \Edge[lw=0.5,opacity=0.2](B)(F)
        \Edge[lw=0.5,opacity=0.2](B)(G)
        \Edge[lw=0.5,opacity=0.2](B)(H)
        \Edge[lw=0.5,opacity=0.2](B)(I)
        \Edge[lw=0.5,opacity=0.2](B)(L)
        \Edge[lw=0.5,opacity=0.2](C)(D)
        \Edge[lw=0.5,opacity=0.2](C)(E)
        \Edge[lw=0.5,opacity=0.2](C)(F)
        \Edge[lw=0.5,opacity=0.2](C)(G)
        \Edge[lw=0.5,opacity=0.2](C)(I)
        \Edge[lw=0.5,opacity=0.2](D)(F)
        \Edge[lw=0.5,opacity=0.2](D)(H)
        \Edge[lw=0.5,opacity=0.2](D)(I)
        \Edge[lw=0.5,opacity=0.2](D)(G)
        \Edge[lw=0.5,opacity=0.2](D)(L)
        \Edge[lw=0.5,opacity=0.2](E)(D)
        \Edge[lw=0.5,opacity=0.2](E)(F)
        \Edge[lw=0.5,opacity=0.2](E)(G)
        \Edge[lw=0.5,opacity=0.2](E)(H)
        \Edge[lw=0.5,opacity=0.2](F)(G)
        \Edge[lw=0.5,opacity=0.2](F)(H)
        \Edge[lw=0.5,opacity=0.2](F)(I)
        \Edge[lw=0.5,opacity=0.2](G)(H)
        \Edge[lw=0.5,opacity=0.2](G)(I)
        \Edge[lw=0.5,opacity=0.2](G)(L)
        \Edge[lw=0.5,opacity=0.2](H)(I)
        \Edge[lw=0.5,opacity=0.2](H)(L)
    \end{tikzpicture}
    \caption{A random walker performing long-range transitions on the graph $G$, whose edges are highlighted by a thick line, can move in one step from node $x_1$ to node $x_7$ with a non-zero transition probability $p_{1,7}$. It means that there exists an edge connecting them directly in a graph $G'$ in which $G$ is a subgraph, in this case the complete graph $G'$ for which the edges are depicted in a thin gray line. The long-range dynamics is nothing more than a walker moving on a graph $G'$ in which $G$ is a subgraph, thus giving the appearance that a far away node could be connected without altering the topology of $G$.}
    \label{fig:overgraph}
\end{figure}

A natural question arises: are the dynamics on the {\textquotedblleft old\textquotedblright} walks along the edges of $G$ compatible with the new dynamics? Indeed, it would be desirable to introduce long-range jumps but preserving at the same time the original dynamics if we move along the edges of $G$. In other words, for any time-interval where does not take place any long-range jump, a random walk on $G'$ should be indistinguishable from the original random walk on $G$. One can easily figure this by a simple but clarifying example: let us suppose that our random walker is surfing the Net (the original graph $G$), and just for the sake of simplicity let us suppose that the Net is undirected. The walker then can move towards linked web-pages with a probability which can be both uniform on the number of total links or dependent on some other parameters. Suppose now that we allow the walker to jump from one web-page to non-linked web-pages by just typing an URL address in the navigation bar, so that he can virtually reach directly any possible web-pages on the Net (the induced graph $G'$). If in any moment, for any reason, the walker is forced again to surf the Net by just following the links, then we should see him moving exactly as he used to do, namely the probability he moves to the next linked web-page has to be the same as before. {More generally, we can suppose that we have a particle freely moving over a surface. There are experimental circumstances under which the particle does not obey a Brownian motion but makes instead ``long jumps''~\cite{Schunack2002,Wrigley1990}. To take into account this fact we change our model to a non-local one, nevertheless at a certain point in time we observe the particle reverting to the Brownian motion: the circumstances have changed. Our model needs now to accommodate both cases by having that the conditions imposed by the unknown change are respected.}

Unfortunately, in general, the induced complete graph $G'$, defined accordingly to the proposal in the literature, breaks the compatibility and the new models cease to be expressions of the original model $G$, see Example \ref{example:cycleembedding}. In Section \ref{sec:comp&embedding} we provide a rigorous definition of what we mean for \emph{compatibility} which stems from a probabilistic interpretation. As it often happens, the property of two graphs of being compatible is the other side of a pure analytic point of view which stems instead from semi-groups and evolution equations theory. This leads to the definition of \emph{embedding} between graphs, and in Theorem \ref{thm:SE&Embedding_equivalence} we prove the equivalence of compatibility and embedding. To avoid misunderstandings, the notion of embedding that we will use from Section~\ref{sec:comp&embedding} onward, it is not the standard one used in topological graph theory but rather it is close to a generalization of the notion of embedding used in differential geometry, although no references to $\R^d$ are made. Notably, equation \eqref{thm:embedding_eq1.2'} in Theorem \ref{thm:embedding} provides an effective computational way to determine whether two graphs are compatible/embeddable or not, which allows to find several practical examples, see Subsection \ref{ssec:counterexamples} and Section \ref{sec:numerical_examples}.     

We are then left in the difficult predicament of trying to add anomalous diffusion dynamics on a given model graph and at the same time maintaining compatibility with the original dynamics. In Section \ref{sec:regularized_fractional_graph-Laplacian}, we show that such an agreement between anomalous diffusion and compatibility can be achieved on any starting graph $G$ by a simple regularization technique on the complete graph $G'$. That allows preserving both the original dynamics of $G$ and the nice properties granted by $G'$. We conclude our analysis with a rich selection of numerical experiments, of both toy-model examples and real-world data. 

The paper is organized as follows:
\begin{itemize}
	\item Section \ref{sec:notation} is devoted to fix the notation and to recall the main known results, and it can be skipped by a reader proficient in graph theory and random walks.
	\item In Section \ref{sec:path&fractional_laplacians} we present the \emph{path} and \emph{fractional} graph Laplacians in the more general setting of weighted graphs. Those operators induce new complete graphs $G_\alpha$ and $G^\alpha$, respectively, which can be seen as supergraphs of $G$.
	\item In Section \ref{sec:comp&embedding} are presented the properties of \emph{compatibility} and \emph{embedding} between two graphs, along with several examples.
	\item In Section \ref{sec:regularized_fractional_graph-Laplacian} is provided a way to regularize the graphs $G_\alpha, G^\alpha$. 
	\item Section \ref{sec:numerical_examples} is the conclusive part of the paper, where several numerical examples are provided to confirm experimentally the theory developed in the preceding sections.
	\item Conclusions are drawn in Section \ref{sec:conclusions}.
\end{itemize}

\section{Setting the notation}\label{sec:notation}
For the reader convenience, Table \ref{tab:symbols} lists an overview of the symbols we will use through out the paper. For a proper and detailed introduction to graph theory see \cite{book_Radek} and \cite{Bollobas}, although we point out that in this work we are going to use a slight different notation.

\begin{table}%
\begin{center}\caption{List of symbols.}
	\begin{tabular}{llrr}\label{tab:symbols}
		$\boldsymbol{1}$ & the counting measure& $G$ & a graph equipped with the counting measure\\
		$\deg$ & the degree function&$\hat{G}$ & a graph equipped with the degree measure\\
		$\alpha\in (0,1)$ & fractional parameter&$G^\alpha$ & the fractional graph associated to $G$\\
		$\alpha\in (1,\infty)$ & path parameter&$G_\alpha$ & the path graph associated to $G$\\
		& & $\prescript{}{r}{G}^\alpha, \prescript{}{r}{G}_\alpha$ & regularized fractional/path graph\\
		& & $\Delta$, $\hat{\Delta}$ & graph Laplacian associated to $G$, $\hat{G}$\\
		& & $\Delta^\alpha, \Delta_\alpha$ & graph Laplacian associated to $G^\alpha, G_\alpha$\\
		& & $\hat{\Delta}^\alpha, \hat{\Delta}_\alpha$ & graph Laplacian associated to $\hat{G}^\alpha, \hat{G}_\alpha$\\
	& & $\prescript{}{r}{\hat{\Delta}}^\alpha, \prescript{}{r}{\hat{\Delta}}_\alpha$ & graph Laplacian associated to $\prescript{}{r}{\hat{G}}^\alpha, \prescript{}{r}{\hat{G}}_\alpha$
	\end{tabular}
\end{center}
\end{table}

\subsection{Graph}  we define \textit{graph}  the triple $G=(X,w,\mu)$, where
\begin{itemize}
	\item $X=\left\{x_i : i\in I\right\}$ is the set of nodes and $I$ is a finite set of indices; 
	\item $w: X\times X \to [0,\infty)$ is a nonnegative edge-weight function; 
	\item $\mu : X \to (0,\infty)$ is a positive (atomic) measure over $X$.
\end{itemize}
 We will assume the following properties to hold:
\begin{enumerate}[label={\upshape(\bfseries G\arabic*)},wide = 0pt, leftmargin = 3em]
	\item\label{assumption:symmetry} $w$ is symmetric, i.e.,  for every couple of nodes $x_i,x_j$ we have $w(x_i,x_j)=w(x_j,x_i)$. 
	\item\label{assumption:loops} no loops allowed: $w(x_i,x_i)=0$ for every $x_i \in X$.
\end{enumerate}
A graph that satisfies \ref{assumption:symmetry} is said to be \emph{undirected}. If $w(x_i,x_j) \in \{0,1\}$ for any pair $\{x_i,x_j\}$, then the graph is said to be unweighted, otherwise it will be called weighted. An undirected and unweighted graph is said \emph{simple}, and a graph such that $w(x_i,x_j)>0$ for every pair $x_i\neq x_j$ is called \emph{complete}. Every unordered pair of nodes $e=\{x_i,x_j\}$ such that $w(x_i,x_j)> 0$ is called \emph{edge} incident to $x_i$ and to $x_j$, and the collection $E$ of all the edges is uniquely determined by $w$. The non-zero values $w(x_i,x_j)$ of the edge-weight function $w$ are called \emph{weights} associated with the edges $\{x_i,x_j\}$.
In an undirected graph $G$, two nodes $x_i,x_j$ are said to be \textit{neighbors} (or \emph{connected}) in $G$ if $\{x_i,x_j\}$ is an edge and we write $x_i\sim x_j$. On the contrary, if $\{x_i,x_j\}$ is not an edge, we  write $x_i\nsim x_j$. We call \emph{walk} a (possibly infinite) sequence of nodes $(x_{i_k} \, : \, k=1,2,\ldots )$ such that $x_{i_k}\sim x_{i_{k+1}}$. A graph is \emph{connected} if there is a finite walk connecting every pair of nodes, that is, for any pair of nodes $x_i, x_j$ there is a finite walk such that $x_i = x_{i_1}\sim x_{i_2}\sim\cdots\sim x_{i_n}=x_j$. We will indicate with $A(G)$ the \emph{adjacency matrix} of $G$, that is, the matrix whose entries are given by the values of $w(\cdot,\cdot)$, i.e., 
$$
\left(A(G)\right)_{i,j}= w(x_i,x_j).
$$
The \emph{degree} of a node $x_i$ is given by
\begin{equation*}\label{def:deg}
\deg(x_i) = \sum_{x_j\in X} w(x_i,x_j).
\end{equation*}
We stress that the function $\deg(\cdot)$ depends on both the node-set $X$ and the edge-weight function $w$. We will indicate with $D(G)$ the diagonal matrix operator whose diagonal entries are given by the values of $\deg(\cdot)$, i.e., 
$$
\left(D(G)\right)_{i,i}= \deg(x_i),
$$
and we call it \emph{degree matrix} of $G$.

\subsection{Graph Laplacian} The set of real functions on $X$ is denoted by $C(X)$, and clearly $C(X)$ is isomorphic to $\R^n$, where $n=\#\{x_i \in X\}$. The operator  $\Delta: C(X) \to C(X)$, associated to the graph $G=(X,w,\mu)$, whose action is defined via
\begin{subequations}
	\begin{equation*}
	\Delta[u](x_i):= \frac{1}{\mu(x_i)}\sum_{\substack{x_j \in X}} w(x_i,x_j)\left(u(x_i) - u(x_j)\right),\label{formal_laplacian2}
	\end{equation*}
\end{subequations} 
is called (nonnegative) \emph{graph Laplacian} of $G$. In matrix form, it reads
\begin{equation}\label{def:graph-Laplacian_matrix_form}
\Delta = \textnormal{diag}(\mu)^{-1}\left[ D(G)-A(G)\right],
\end{equation}
where $\textnormal{diag}(\mu)$ is the diagonal matrix operator whose diagonal entries are given by the pointwise evaluation of $\mu$. The graph Laplacian is selfadjoint with respect to the inner product on $C(X)$
$$
\langle u, v \rangle := \sum_{x_i \in X} u(x_i)v(x_i)\mu(x_i).
$$
Since non-local dynamics for long-range jumps (see Section \ref{sec:path&fractional_laplacians}) were introduced for application purposes, and in order to not excessively digress, we will make the following extra assumptions:
\begin{enumerate}[label={\upshape(\bfseries G\arabic*)},wide = 0pt, leftmargin = 3em]
	\setcounter{enumi}{2}
	\item\label{assumption:connected} $G$ is connected;
	\item\label{assumption:measure} $\mu$ is restricted to be the counting measure $\boldsymbol{1}(x_i)\equiv 1$ or the degree measure, that is
	$$
 \mu = \boldsymbol{1} \quad\mbox{or}\quad  \mu=\deg.
	$$
\end{enumerate}
From Assumption \ref{assumption:connected} we have that $\deg$ is indeed a positive measure on $X$.

\subsection{Unnormalized and normalized graph} If not otherwise stated, $G$ will be equipped with the counting measure and in that case we call it \emph{unnormalized}. On the contrary, we will use the notation $\hat{G}$ to indicate a graph equipped with the degree measure, and we will call it \emph{normalized}. We will keep the notation $G=(X,w,\mu)$ whenever we do not specify whether the measure $\mu$ is the counting measure or the degree measure. Given a graph $G=(X,w,\boldsymbol{1})$ then we will call \emph{normalization} of $G$ the graph $\hat{G}=(X,w,\deg)$, that is, the graph given by the same node-set and edge-weight function of $G$ but equipped with the degree measure. We will use the notation $\hat{\Delta}$ to indicate the graph Laplacian associated to a normalized graph $\hat{G}$. Let us trivially observe that, if $\hat{G}$ is the normalization of $G$, then $D(\hat{G})=D(G)$ and $A(\hat{G})=A(G)$. Therefore, from \eqref{def:graph-Laplacian_matrix_form}, in matrix form the graph Laplacian of an unnormalized graph and of its normalization, respectively, read 
\begin{equation}\label{def:graph-Laplacian_matrix_form2}
\Delta =  D(G)-A(G), \qquad \hat{\Delta}= \operatorname{id} -  D(G)^{-1}A(G),
\end{equation}
where $\operatorname{id}: C(X) \to C(X)$ is the identity operator, $\operatorname{id}[u]=u$. 

\subsection{Subgraph}\label{ssec:subgraph}
We say that a general graph $G=(X, w, \mu)$ is a \emph{subgraph} of $G'=(X',w',\mu')$, and we write $G\subset G'$, if
\begin{itemize}
	\item $X\subseteq X'$;
	\item $w : X \times X \to [0, +\infty)$ is such that $w(x_i,x_j) \in \{0, w'(x_i,x_j)\}$;
	\item $\mu := \mu'_{|(X,w)}$.
\end{itemize} 
We call $G'$ a \emph{supergraph} of $G$. If $w=w'_{|X\times X}$, then it is said that $G$ is an \emph{induced} (or \emph{spanned}) subgraph of $G'$. Let us observe that, even if $X=X'$, in general $\deg'_{|(X,w)}\neq \deg'$.

\subsection{Random walk}\label{ssec:random_walk}
Given a graph $G=(X,w,\boldsymbol{1})$, a \emph{random walk} on $G$ is a sequence of discrete steps along the nodes $x_i$ which abides by the following rules: starting from a node $x_k$, we choose at random the next node where to move among the neighbors of $x_k$, with a probability proportional to the weights associated to the edges connecting the nodes; once we moved, we repeat the procedure.

More rigorously, a random walk on $G$ with initial distribution $\nu$ ({i.e., $\nu$ a nonnegative, finite measure on the powerset $2^X$ and such that $\sum_{k\in I} \nu(\{x_k\})=1$}) is a discrete-time Markov chain $(X_n)_{n\geq0}=(X_n \, : \, n \in \N)$ on the state-space $X$ such that, for every fixed $x_k \in X$ 
\begin{enumerate}[label={\upshape(\bfseries MC\arabic*)},wide = 0pt, leftmargin = 3em]
	\item\label{MC1} $\mathbb{P}(X_0=x_{k})= \nu(\{x_{k}\})$;
	\item\label{MC2} $\mathbb{P}(X_{n+1}=x_{j} \,|\, X_0=x_{k}, \ldots, X_{n}=x_{i})= \frac{w(x_{i},x_{j})}{\deg(x_{i})}$.
\end{enumerate}  
Specifically, under the assumptions \ref{MC1}-\ref{MC2} and \ref{assumption:symmetry}-\ref{assumption:connected}, $(X_n)_{n\geq0}$ is a homogeneous, irreducible and reversible Markov chain, characterized by the \emph{transition  matrix} $D(G)^{-1}A(G)$. For a proper introduction to Markov chains theory we refer to \cite{Norris1997,Durrett2010}. It is a habit to write
\begin{align*}\label{def:MC_transition_matrix}
&P:=D(G)^{-1}A(G), &(P)_{i,j}:= p_{i,j}= \frac{w(x_{i},x_{j})}{\deg(x_{i})},\\
&P(n):=P^n, & \left(P(n)\right)_{i,j}:=p_{i,j}^{(n)}.
\end{align*}
The entries $p_{i,j}$ are called (discrete) \emph{transition probabilities}. A \emph{continuous} random walk $(X_t)_{t\geq 0}=(X_t \, : \, t\in \R_+)$ can be though instead as a random walk whose times spent on each node are not homogeneous anymore but independent, negative exponential distributed random variables. Its transition matrix $P(t)$ is determined by the minimal nonnegative solution of the backward equation
\begin{equation*}
(P(t))_{i,j}:= p_{i,j}(t) \quad \mbox{sucht that} \quad\begin{cases}
\frac{d P(t)}{dt} + \hat{\Delta}P(t)=0,\\
P(0)=\operatorname{id},
\end{cases}
\end{equation*} 
where $\hat{\Delta}$ is the graph Laplacian associated to the normalization $\hat{G}$ of $G$. Therefore, given an initial distribution $\nu$ and the transition matrix $P$ or $P(t)$, the dynamics of a random walker on $G$ is well established. Since, from equation \eqref{def:graph-Laplacian_matrix_form2}, it holds that $\hat{\Delta}= \operatorname{id} - P$, then both $P$ and $P(t)$ are uniquely determined by the edge-weight function $w(\cdot,\cdot)$, and we say that $w$ \emph{induces the dynamics} of the random walks on $G$.
\begin{remark}
In this work we consider as time-continuous random walks on a given graph $G$ the standard \emph{continuation} of the discrete-time random walks. However, there exists another possible way which consists in taking the transition matrix $P(t):=\textrm{e}^{-t\Delta}$, where $\Delta$ is the graph Laplacian of $G$ itself. This approach goes by the name of \emph{fluid model} and even if in the case of an unweighted regular graph the two approaches are identical up to a deterministic time rescaling factor, it is not true in general. See \cite[Section 3.2]{DAldous}.        
\end{remark}

\section{Path and fractional graph Laplacians}\label{sec:path&fractional_laplacians}
In this section we generalize the ideas developed in \cite{Estrada,RM12,RM14} to the more general setting of weighted graphs. For two examples of a path graph Laplacian and a fractional graph Laplacian, look at Counterexample \ref{cexample:dpathlaplacian} and Counterexample \ref{example:1}, respectively.	
\subsection{The path graph Laplacian}\label{ssection:path_Laplacian}
Let us fix the functions $\delta, d_w: X\times X \to [0, +\infty)$, where
\begin{align*}
&\delta(x_i,x_j):= \min\left\{ n \in \N \, : \, x_i=x_{i_1}\sim x_{i_2} \sim \ldots \sim x_{i_n}= x_j \right\}
\end{align*}	
is the \emph{combinatorial graph distance} and $d_w$ is instead a generic distance function, but dependent to $w(\cdot,\cdot)$ as well. For a review on path pseudo-distances and intrinsic metrics, see \cite[Chapter 3]{Keller} and all the references therein. We call 
$$
\delta_\infty := \max\left\{ \delta(x_i,x_j) \, : \, x_i,x_j \in X  \right\}
$$
the \emph{combinatorial diameter} of $G$. Given a one parameter family of functions $h_\alpha : \R_+ \to (0,+\infty)$, $\alpha \in (0,+\infty)$, we can define the following operator,
\begin{equation}\label{eq:path-graph-operator}
\Delta_{\alpha} :=  \sum_{n=1}^{\delta_\infty} \Delta_{\alpha,n}, \qquad \mbox{where}\qquad \Delta_{\alpha,n}[u](x_i):=  \sum_{x_j\in X} \kappa_{\alpha,n}(x_i,x_j)(u(x_i) - u(x_j)) 
\end{equation}
and
\begin{equation}\label{eq:path-graph}
\kappa_{\alpha,n}(x_i,x_j):= \begin{cases}
h_\alpha(d_w(x_i,x_j)) & \mbox{if } \delta(x_i,x_j)=n,\\
0& \mbox{otherwise}.
\end{cases}
\end{equation}
Let us observe that whenever $G$ is unweighted, then fixing $d_w=\delta$ and
$$
h_\alpha(t):= t^{-\alpha} \qquad \mbox{or} \qquad h_\alpha(t):= \textrm{e}^{-\alpha t}
$$
we retrieve the Mellin transformed and the Laplace transformed path graph Laplacians, respectively, introduced by E. Estrada. See \cite{Estrada,RM12,EHHL,EHLP}, and \cite{EDHMMRS,CDT} for the most recent developments. We will call $\Delta_{\alpha}$ a \emph{path} graph Laplacian. It defines in a natural way a new complete graph $G_\alpha :=\left(X, w_\alpha, \boldsymbol{1}\right)$, where the edge-weight function can be expressed by
$$
w_\alpha (x_i,x_j) := \sum_{n=1}^{\delta_\infty}\kappa_{\alpha,n}(x_i,x_j).
$$
Another way to express the above formulation, is to consider the matrices $\tilde{A}(G)$ and $K_{\alpha,n}$ defined by
$$
\left(\tilde{A}(G)\right)_{i,j}:=\begin{cases}
1 & \mbox{if } w(x_i,x_j)\neq 0,\\
0 & \mbox{otherwise,}
\end{cases}\qquad
\left(K_{\alpha,n}\right)_{i,j}:= \kappa_{\alpha,n}(x_i,x_j).
$$
Then, the adjacency matrix associated to $G_\alpha$ can be written as
$$
A(G_\alpha)= \sum_{n=1}^{\delta_\infty} \left(\tilde{A}(G)\right)^n \odot K_{\alpha,n}, 
$$
where $\odot$ is the Hadamard product. We call $G_\alpha$ the \emph{path graph} of $G$.
	
\subsection{The fractional graph Laplacian}\label{ssec:fractional_graph_laplacian}
Since $\Delta$ is a selfadjoint operator, it is possible to define the \emph{fractional} graph Laplacian operator $\Delta^\alpha$, with $\alpha\in (0,1)$, such that 
$$
\Delta= U\Lambda U^T \qquad \mapsto \qquad			\Delta^\alpha := U\Lambda^\alpha U^T,
$$
where $(U,\Lambda)$ is the spectral decomposition of $\Delta$ and $U^T$ is the transpose of $U$. The scheme is the following: 
$$
(G,\Delta) \mapsto \Delta^\alpha \mapsto (G^\alpha,\Delta^\alpha).
$$
That is, we start from a weighted graph $G=(X,w,\boldsymbol{1})$ and its associated graph Laplacian $\Delta$. For any fixed $\alpha \in (0,1)$ we get a selfadjoint operator $\Delta^\alpha$ which is the graph Laplacian associated to a new graph $G^\alpha=(X,w^\alpha,\boldsymbol{1})$. For an account about fractional graph Laplacian on simple graphs and their applications we refer to \cite{RM14,Riascos2015,Michelitsch2017,DeNigris2017,Riascos2018,Michelitsch2019,Bautista2019,BBDS}. We also note in passing that this construction can be extended to cover the case of directed graphs by moving from the spectral decomposition of $\Delta$ to the Jordan canonical form, we refer to~\cite{BBDS} for the complete details. We avoid here this further degree of generality since the main point of the discussion in Section~\ref{sec:comp&embedding} can be extended transparently to the oriented case.

We have the following result that generalizes \cite[Proposition 3.3, Corollary 3.5]{BBDS}.
\begin{proposition}\label{prop:fractional_decay_rate}
For any given $\alpha \in (0,1)$ and a fixed graph $G=(X,w, \boldsymbol{1})$ along with its associated graph Laplacian $\Delta$, there exists a unique graph $G^\alpha=(X,w^\alpha, \boldsymbol{1})$ such that $\Delta^\alpha$ is the graph Laplacian associated to $G^\alpha$. Moreover, $G^\alpha$ is a complete graph and it holds that
\begin{equation*}
0<w^\alpha(x_i,x_j) \leq c\left(\frac{\rho(\Delta)}{2(\delta(x_i,x_j)-1)}\right)^\alpha \quad \forall \, x_i,x_j \mbox{ such that } \delta(x_i,x_j)>1 \mbox{ in } G,
\end{equation*}
where $c$ is a constant independent from the nodes, $\rho(\Delta)$ is the spectral radius of $\Delta$ and $\delta(\cdot,\cdot)$ is the combinatorial distance function on $G$.
\end{proposition}
\begin{proof}
Since $(\Delta)_{i,i}= \sum_{x_j\in X}w(x_i,x_j)=-\sum_{j\neq i}(\Delta)_{i,j}$ and $G$ is connected, then it is immediate to check that we can write
$$
\Delta = s\cdot\operatorname{id} - B, \qquad s:= \max_{i=1,\ldots,n} \left\{(\Delta)_{i,i}\right\}> 0
$$
and $B$ a nonnegative matrix. Therefore, by the boundedness of $s^{-1}B$, it holds that
$$
\Delta^\alpha = s^\alpha \left( \operatorname{id} - s^{-1}B\right)^\alpha = s^\alpha\sum_{k=0}^\infty (-1)^k\binom{\alpha}{k}(s^{-1}B)^k.
$$
Observe now that, as a consequence of the connectedness  of $G$, for every $i\neq j$ there exists $k\geq 1$ such that $(B^k)_{i,j}\neq 0$, and since $(-1)^k\binom{\alpha}{k}<0$ for every $\alpha\in (0,1),k\geq 1$, then we conclude that $(\Delta^\alpha)_{i,j}<0$ for every $i\neq j$. On the other hand, writing $u_{\boldsymbol{1}}$ for the unit-constant function, that is, $u_{\boldsymbol{1}}(x_i)=1$ for every $x_i \in X$, then again by standard arguments it holds that $\Delta^\alpha\left[u_{\boldsymbol{1}} \right]=0$.  Therefore, $(\Delta^\alpha)_{i,i}= -\sum_{j\neq i}(\Delta^\alpha)_{i,j}$ and then by definition $\Delta^\alpha$ determines uniquely a complete graph $G^\alpha=(X,w^\alpha, \boldsymbol{1})$, such that $w^\alpha(x_i,x_j)=-(\Delta^\alpha)_{i,j}\neq 0$ for every $i\neq j$. Clearly, the graph Laplacian associated to $G^\alpha$ coincides with $\Delta^\alpha$. The last part of the proof can be derived exactly by the same arguments in \cite[Proposition 3.3 and Corollary 3.5]{BBDS}. The crucial point is to observe that $f(s)=s^\alpha$ over $[0,1]$ is $\alpha$-H\"{o}lder and has modulus of continuity $\omega_f(s)=s^\alpha$. Then, by the Jackson Theorem it can be proved that
$$
\left|\left(\Delta^\alpha\right)_{i,j}\right|\leq  (1+\pi^2/2)\omega_f \left( \frac{1}{2(\delta(x_i,x_j)-1)}\right).
$$
\end{proof}
\noindent We call $G^\alpha$ the \emph{fractional graph} of $G$ .

\section{Compatibility and Embedding}\label{sec:comp&embedding}
This is the core of the paper. In Section \ref{sec:path&fractional_laplacians} we described how to build the path and the fractional graph $G_\alpha=(X,w_\alpha,\boldsymbol{1})$ and $G^\alpha=(X,w^\alpha,\boldsymbol{1})$, respectively, starting from a graph $G=(X,w,\boldsymbol{1})$. The main question is whether $G_\alpha, G^\alpha$ are {\textquotedblleft compatible\textquotedblright} with $G$. Basically, we added new edges to the graph $G$ allowing the random walker to make long-range jumps which were not originally permitted. If, for any reasons, in any time-interval the random walker is deprived of those long-range jumps, or we know that he moved only along the old edges of $G$, then if our model is compatible we should not be able to distinguish the dynamics of the walker on $G_\alpha,G^\alpha$ from its original dynamics on $G$. This is explained in Definition \ref{def:SE} and Proposition \ref{prop:conditioned_random_walk}. It is indeed desirable that $G_\alpha$ and $G^\alpha$ preserve the original dynamics of $G$, without breaking it, otherwise our new models would not be anymore expressions of the original model $G$.

Noteworthy, as it usually happens, there is an interplay between the probabilistic and the pure analytic point of views which materializes in Theorem \ref{thm:SE&Embedding_equivalence}, intertwining the definition of \emph{compatibility} in \ref{def:SE} with the definition of \emph{embedding} in \ref{def:embedding&dynamics}. This provides us with an effective computational way to understand if a graph $G'$ is compatible with another graph $G$.

As we will see in Subsection \ref{ssec:counterexamples}, unfortunately the parameter $\alpha$ breaks the original dynamics, in general. To make this rigorous, we introduce a couple of preliminary notations.

\begin{definition}[Pullback]
Let $\phi: X \to Y$ be a bijection. Given a function $u\in C(X)$, we call \emph{pullback} of $u$ the function $u^*:=u\circ\phi^{-1} \in C(Y)$. 
\end{definition}
{Let us observe that the above definition of pull-back extends naturally to measures. That is, given a measure $\nu$ on the measure space $(X,2^X)$, then its pull-back $\nu^*=\nu\circ \phi^{-1}$ is a measure on the measure space $(Y,2^Y)$.}
\begin{definition}[$\phi$-induced subgraph]\label{def:phi-subgraph}
Given two graphs $G=(X,w,\mu)$ and $G'=(X', w',\mu')$, let $\phi : X \to \phi(X)\subseteq X'$ be a bijection onto its image $\phi(X)$ that satisfies the property
\begin{equation}\label{hp:embedding_map1}
\tag{\textbf{E1}}\mbox{for every edge } \{x_i,x_j\} \mbox{ of }G,\; \{\phi(x_i),\phi(x_j)\} \mbox{ is an edge of } G'.
\end{equation}
We call \emph{$\phi$-induced subgraph} of $G'$ the subgraph $\phi(G):=\left(\phi(X),  w_\phi, \mu_\phi\right)\subset G'$ such that
\begin{itemize}
		\item $\phi(X)=\{ \phi(x_i) \, : \, x_i \in X  \}$;
		\item $w_\phi(x'_i,x'_j)= \begin{cases}
		w'(x'_i, x'_j) & \mbox{if } \{\phi^{-1}(x'_i),\phi^{-1}(x'_j)\}\mbox{ is an edge of }G,\\
		0& \mbox{otherwise};
		\end{cases}$
		\item $\mu_\phi = \mu'_{|(\phi(X), w_\phi)}$.
	\end{itemize}
\end{definition}
Let us point out that even if $X\subseteq X'$ and $\phi \equiv \operatorname{id}$, where $\operatorname{id}: X \to X\subseteq X'$ is the identity map, in general $G\neq \operatorname{id}(G)$, since in general $w \neq w_{\operatorname{id}}$ and then $G$ could not be a subgraph of $G'$. For a visual representation of a $\phi$-subgraph, look at Figure~\ref{fig:embedding}.
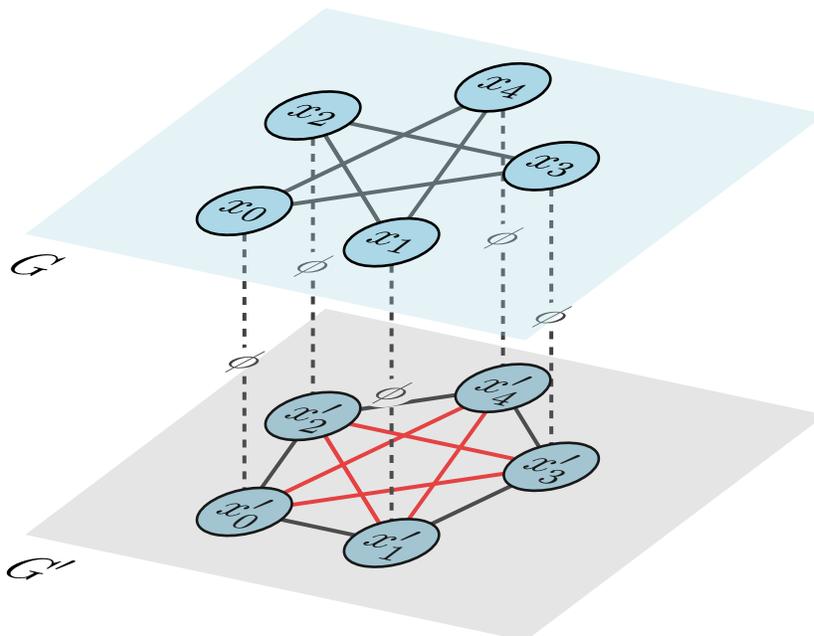
\begin{figure}
	\centering
	\begin{tikzpicture}[multilayer=3d,scale=2]
	\begin{Layer}[layer=1]
	\Plane[x=-.6,y=-.5,width=2.5,height=3.4,NoBorder]
	\node at (-.5,-.6)[below right]{\footnotesize $G$};
	\Vertex[x=0.5,size=.5,label=x_0,Math]{A}
	\Vertex[x=1.5,size=.5,label=x_1,Math]{B}
	\Vertex[x=0.19,y=0.95,size=.5,label=x_2,Math]{C}
	\Vertex[x=1.81,y=0.95,size=.5,label=x_3,Math]{D}
	\Vertex[x=1,y=1.54,size=.5,label=x_4,Math]{E}
	\Edge(A)(D) 
	\Edge(A)(E)
	\Edge(B)(C)
	\Edge(B)(E)
	\Edge(D)(C)
	\Edge(C)(B)
	\end{Layer}
	\begin{Layer}[layer=2]
	\node at (-.5,-.6)[below right]{\footnotesize $G'$};
	\Vertex[x=0.5,size=.5,label=x_0',Math]{A'}
	\Vertex[x=1.5,size=.5,label=x_1',Math]{B'}
	\Vertex[x=0.19,y=0.95,size=.5,label=x_2',Math]{C'}
	\Vertex[x=1.81,y=0.95,size=.5,label=x_3',Math]{D'}
	\Vertex[x=1,y=1.54,size=.5,label=x_4',Math]{E'}
	\Edge[color=red!80](A')(D')
	\Edge[color=red!80](A')(E')
	\Edge[color=red!80](B')(C')
	\Edge[color=red!80](B')(E')
	\Edge[color=red!80](D')(C')
	\Edge[color=red!80](C')(B')
	\Edge(A')(B')
	\Edge(B')(D')
	\Edge(D')(E')
	\Edge(E')(C')
	\Edge(C')(A')
	\Plane[x=-.6,y=-.5,width=2.5,height=3.4,NoBorder,color=gray,opacity=0.2]
	\end{Layer}
	\Edge[style=dashed,Math,label=\phi](A)(A')
	\Edge[style=dashed,Math,label=\phi](B)(B')
	\Edge[style=dashed,Math,label=\phi](C)(C')
	\Edge[style=dashed,Math,label=\phi](D)(D')
	\Edge[style=dashed,Math,label=\phi](E)(E')
	\end{tikzpicture}
	\caption{Projection of the graph $G=(X,w,\mu)$ into $G'=(X',w',\mu')$ by the bijection $\phi(x_i)=x_i'$. In red are depicted the common edges shared by $G'$ and $\phi(G)$.}
	\label{fig:embedding}
\end{figure}
\subsection{Compatibility}
In this subsection we give a rigorous probabilistic interpretation for the way of saying {\textquotedblleft to have the same dynamics\textquotedblright}.
\begin{definition}[Stochastic equivalence and compatibility]\label{def:SE}
	Given two graphs $G$ and $G'$, we say that they are \emph{stochastically equivalent} if there exists a bijection $\psi : I \to I'$ between their index sets such that
	\begin{equation}\label{eq:SE}
	\tag{\textbf{SE}} p_{i,j} = p'_{\psi(i),\psi(j)} \quad \mbox{for every }i,j\in I,
	\end{equation}
	where $p_{i,j}$ and $p'_{\psi(i),\psi(j)}$  are the (discrete) transition probabilities associated to $G$ and $G'$, respectively.
	
	Let now $G$ and $G'$ be such that $X\subseteq X'$ and $E\subseteq E'$, where $E$ is the edge-set of $G$ and $E'$ is the edge-set of $G'$, respectively. If $G$ and $\operatorname{id}(G)$ are stochastically equivalent, then we say that $G'$ is \emph{compatible} with $G$. 
\end{definition}
In other words, the above definition is telling us that, given any initial distribution $\nu$ and its pullback $\nu^*$, the random walk $(X_t)_{t\geq 0}$ on $G$ with initial distribution $\nu$ and the random walk $(X'_t)_{t\geq0}$ on $G'$ with initial distribution $\nu^*$ are \emph{stochastically equivalent in the wide sense} (see \cite[Definition 2 pg. 43]{Gihman2004}), where $t$ can take values in $\N$ or $\R_+$. Indeed, since $\hat{\Delta}= \operatorname{id} - P$ by equation \eqref{def:graph-Laplacian_matrix_form2}, both the probability distributions of the discrete-time and the continuous-time random walks are determined by $P$. This supports the  way of saying {\textquotedblleft to have the same dynamics\textquotedblright} when two graphs $G$ and $G'$ are stochastically equivalent.
\begin{remark*}
The definition of compatibility is well-posed, since if $E\subseteq E'$, then $\operatorname{id}(\cdot)$ satisfies \eqref{hp:embedding_map1} and it trivially defines an identity $\psi : I \to I$ on the index set $I$ of $X$, which is shared by $G$ and $\operatorname{id}(G)$. 
\end{remark*}
We provide now a characterization of random walks on a subgraph $G\subset G'$ as random walks on the supergraph $G'$ conditioned to move only along the edges of $G$. Given a graph $G'$, let us write
\begin{equation*}
\mathcal{E}':= \left\{ (x'_{i_k})_k \, | \, x'_{i_{k}}\sim x'_{i_{k+1}} \mbox{ in } G'  \right\}
\end{equation*} 
for the set of all possible walks on $G'$, with the convention that $(x'_i)$ belongs to $\mathcal{E}'$ for every $x'_i \in X'$. Clearly, $\mathcal{E}'$ is uniquely determined by the edge-weight function $w'$. On the other hand, if we fix a proper subgraph $G\subset G'$, then the edge-weight function $w$ of $G$ defines uniquely a new set $\mathcal{E}$ of all possible walks on $G$ such that $\mathcal{E}\subset \mathcal{E}'$, {with again the convention that $(x'_i) \in \mathcal{E}$ if $x'_i \in X$, where $X$ is the node set of $G$}. Given an initial distribution $\nu'$ on $(X',2^{X'})$, {such that $\nu'(X)>0$}, and its associated random walk $(X'_n)_{n\in \N}$  on $G'$, let us define the following set of finite dimensional distributions {on the state space $X$},
\begin{equation}\label{eq:finite_dim_distr_subgraph}
{\mu_n \left(\{x'_{i_0}\}\times \ldots \times \{x'_{i_n}\} \right):=\prod_{k=0}^n\mathbb{P}\left(X'_{k}=x'_{i_{k}}\,|\, X'_0=x'_{i_0}, \ldots, X'_{k-1}=x'_{i_{k-1}}, (X'_{j})_{j=0}^{k} \in \mathcal{E}  \right),}
\end{equation}
{where $x'_{i_0},\ldots,x'_{i_n} \in X$.} We have the following result, Proposition \ref{prop:conditioned_random_walk}, that highlights the relationship between a random walk on a subgraph $G\subset G'$ and a random walk on $G'$ conditioned to move only on the nodes which are neighbors in $G$. {Let us recall the notations introduced in Subsections \ref{ssec:subgraph} and \ref{ssec:random_walk}, that is,}
$$
{p'(x'_{i}, x'_j) = \frac{w'(x'_i,x'_j)}{\deg'(x'_i)}\; \forall \,x'_i,x'_j\in X', \qquad p(x_{i}, x_j) = \frac{w(x_i,x_j)}{\deg(x_i)}\; \forall\, x_i,x_j\in X\subseteq X',}
$$
{where $w', \deg'$ are the edge-weight and the degree functions of $G'$, respectively, $w$ is the edge-weight function associated to the subgraph $G$ and $\deg(x_i)=\deg'_{|(X,w)}(x_i)$.} 
\begin{proposition}\label{prop:conditioned_random_walk}
{Let $G=(X, w, \mu)$ be a subgraph of $G'=(X', w',\mu')$, and let $(X'_n)_{n\geq0}$ be the Markov chain on $X'$ with initial distribution $\nu'$, $\nu'(X)>0$}. There exists a stochastic process $(\tilde{X}_n)_{n\in\N}$ such that its finite dimensional distributions are given by \eqref{eq:finite_dim_distr_subgraph}. $(\tilde{X}_n)_{n\in\N}$ is stochastically equivalent in the wide sense to the random walk $\left(X_n\right)_{n\in \N}$ defined on $G$ and with initial distribution $\nu$ given by
$$
{\nu(\{x_i\})=\frac{\nu'(\{x_i\})}{\nu'(X)} \quad \mbox{for every } x_i \in X\subseteq X'.}
$$
{In particular, $\mathbb{P}\left(\tilde{X}_{n+1}=x'_{i_{n+1}}\, | \, \tilde{X}_{0}=x'_{i_{0}},\ldots,\tilde{X}_{n}=x'_{i_{n}}\right)$ is equal to} 
    $$
    {\mathbb{P}\left(X'_{n+1}=x'_{i_{n+1}}\, | \, X'_{0}=x'_{i_{0}},\ldots,X'_{n}=x'_{i_{n}}, (X'_k)_{k=0}^{n+1}\in \mathcal{E}\right).}
    $$
\end{proposition}
\begin{proof}
{Let $(X_n')_{n\geq 0}$ be the discrete-time Markov chain on the state space $X'$ with initial distribution $\nu'$ on $2^{X'}$. For notation brevity, let us define 
\begin{align*}
&A_{x'_{i_k}}:= \left\{\omega \in \Omega \,: \, X'_k(\omega)=x'_{i_k} \right\},  \\
&A_{x'_{i_0},\ldots,x'_{i_k}}:= \left\{\omega \in \Omega \,: \, X'_0(\omega)=x'_{i_0}, \ldots,X'_k(\omega)=x'_{i_k} \right\}= \bigcap_{j=0}^k A_{x'_{i_j}},\\
&B_k:= \left\{\omega \in \Omega \,: \, (X'_{j}(\omega))_{j=0}^{k} \in \mathcal{E} \right\},\\
&C_{x'_{i_k}}:=\left\{\omega \in \Omega \, : \, (x'_{i_k}, X'_{k+1}(\omega)) \in \mathcal{E} \right\}= \bigsqcup_{\substack{x'_j\in X:\\(x'_{i_k},x'_j) \in \mathcal{E}}}\left\{\omega \in \Omega \, : \,  X'_{k+1}(\omega)= x'_j \right\},\\
&t(x'_{i},x'_{j}):=\begin{cases}
0 & \mbox{if } x'_{i},x'_{j}\notin X,\\
\mathds{1}_{(0,\infty)}\left(w(x'_{i},x'_{j})) \right) & \mbox{otherwise}.
\end{cases}
\end{align*}
With this notation, we recall that by \ref{MC1}-\ref{MC2} it follows that
\begin{align*}
&\mathbb{P}\left(A_{x'_{i_{k+1}}}\, |\, A_{x'_{i_0},\ldots,x'_{i_{k}}} \right)=\mathbb{P}\left(A_{x'_{i_{k+1}}}\, |\, A_{x'_{i_{k}}} \right)=p'_{i_{k},i_{k+1}},\\
&\mathbb{P}\left(A_{x'_{i_0},\ldots,x'_{i_k}}\right)= \nu'(x'_{i_0})\prod_{j=1}^{k}p'_{i_{j-1},i_{j}},
\end{align*}
see for example \cite[Theorems 1.1.1 and 1.1.3]{Norris1997}. As preliminary remarks, observe that
$$
A_{x'_{i_0},\ldots,x'_{i_{n+1}}}\cap B_{n+1} = \begin{cases}
A_{x'_{i_0},\ldots,x'_{i_{n+1}}} & \mbox{if } (x'_{i_k})_{k=0}^{n+1}\in \mathcal{E}\\
\emptyset & \mbox{otherwise},
\end{cases}
$$
and that
$$
A_{x'_{i_0},\ldots,x'_{i_{n}}}\cap B_{n+1} = \begin{cases}
A_{x'_{i_0},\ldots,x'_{i_{n}}}\cap C_{x'_{i_n}} & \mbox{if } (x'_{i_k})_{k=0}^{n}\in \mathcal{E}\\
\emptyset & \mbox{otherwise}.
\end{cases}
$$
Let now observe that
\begin{align*}
\mathbb{P}\left(A_{x'_{i_0},\ldots,x'_{i_{n}}}, C_{x'_{i_{n}}} \right)&=\mathbb{P}\left(A_{x'_{i_0}},\ldots, A_{x'_{i_{n}}}, C_{x'_{i_{n}}} \right)\\
&=\mathbb{P}\left(A_{x'_{i_0}}\right)\mathbb{P}\left(A_{x'_{i_1}}\, |\, A_{x'_{i_0}}\right)\cdots \mathbb{P}\left(A_{x'_{i_n}}\, |\, A_{x'_{i_0},\ldots,x'_{i_{n-1}}}\right)\mathbb{P}\left(C_{x'_{i_{n}}}\, |\, A_{x'_{i_0},\ldots,x'_{i_{n}}}\right)\\
&=\nu'(x'_{i_0})\left(\prod_{k=1}^{n} p'_{i_{k-1},i_{k}}\right)\mathbb{P}\left(C_{x'_{i_{n}}}\, |\, A_{x'_{i_0},\ldots,x'_{i_{n}}} \right).
\end{align*}
Clearly, if $x'_{i_n}\notin X$, then $C_{x'_{i_{n}}}=\emptyset$ and $\mathbb{P}\left(A_{x'_{i_0},\ldots,x'_{i_{n}}}, C_{x'_{i_{n}}} \right)=0$. Instead, if $x'_{i_n}\in X$, then
\begingroup
\allowdisplaybreaks
\begin{align*}
\mathbb{P}\left(C_{x'_{i_{n}}}\, |\, A_{x'_{i_0},\ldots,x'_{i_{n}}} \right)&=\sum_{\substack{x'_j\in X:\\(x'_{i_n},x'_j)\in\mathcal{E}}}\mathbb{P}\left(X'_{n+1}=x'_j\, |\, A_{x'_{i_0},\ldots,x'_{i_{n}}} \right)\\
&=\sum_{\substack{x'_j\in X:\\(x'_{i_n},x'_j)\in\mathcal{E}}}p'_{i_n,j}\\
&= \sum_{\substack{x'_j\in X:\\(x'_{i_n},x'_j)\in\mathcal{E}}} \frac{w'(x'_{i_n},x'_j)}{\deg'(x'_{i_n})}\\
&= \frac{\deg'_{|(X,w)}(x'_{i_n})}{\deg'(x'_{i_n})}>0,
\end{align*}
where the last inequality is due to the fact that both $G$ and $G'$ are connected (see \ref{assumption:connected}). Therefore, assuming that $x'_{i_0}, \ldots, x'_{i_n}\in X$, from the above equation it holds that
\begin{align}
\mathbb{P}\left( X'_{n+1}=x'_{i_{n+1}} \, | \, A_{x'_{i_0},\ldots,x'_{i_{n}}}, B_{n+1} \right)&= \frac{\mathbb{P}\left(X'_{n+1}=x'_{i_{n+1}},A_{x'_{i_0},\ldots,x'_{i_{n}}}, B_{n+1} \right)}{\mathbb{P}\left(A_{x'_{i_0},\ldots,x'_{i_{n}}}, B_{n+1} \right)} \nonumber\\
&= \frac{\mathbb{P}\left(A_{x'_{i_0},\ldots,x'_{i_{n+1}}}, B_{n+1} \right)}{\mathbb{P}\left(A_{x'_{i_0},\ldots,x'_{i_{n}}}, B_{n+1} \right)}\nonumber\\
&=\frac{\mathbb{P}\left(A_{x'_{i_0},\ldots,x'_{i_{n+1}}} \right)t(x'_{i_n},x'_{i_{n+1}})}{\mathbb{P}\left(A_{x'_{i_0},\ldots,x'_{i_{n}}}, C_{x'_{i_{n}}} \right)}\nonumber\\
&= \frac{\nu'(x'_{i_0})\prod_{k=1}^{n+1} p'_{i_{k-1},i_{k}}t(x'_{i_n},x'_{i_{n+1}})}{\nu'(x'_{i_0})\prod_{k=1}^{n} p'_{i_{k-1},i_{k}}\mathbb{P}\left(C_{x'_{i_{n}}}\, |\, A_{x'_{i_0},\ldots,x'_{i_{n}}} \right)}\nonumber\\
&=\frac{ p'_{i_{n},i_{n+1}}t(x'_{i_n},x'_{i_{n+1}})}{\frac{\deg'_{|(X,w)}(x'_{i_n})}{\deg'(x'_{i_n})}}\nonumber\\
&= \frac{\frac{w'(x'_{i_n},x'_{i_{n+1}})}{\deg'(x'_{i_n})} t(x'_{i_n},x'_{i_{n+1}})}{\frac{\deg'_{|(X,w)}(x'_{i_n})}{\deg'(x'_{i_n})}}\nonumber\\
&= \frac{w'(x'_{i_n},x'_{i_{n+1}})}{\deg'_{|(X,w)}(x'_{i_n})}t(x'_{i_n},x'_{i_{n+1}})\nonumber\\
&= \begin{cases}
\frac{w(x'_{i_n},x'_{i_{n+1}})}{\deg'_{|(X,w)}(x'_{i_n})} &\mbox{if } x'_{i_{n+1}} \in X,\\
0 & \mbox{otherwise},
\end{cases}\nonumber\\
&= \begin{cases}
p_{i_n,i_{n+1}} &\mbox{if } x'_{i_{n+1}} \in X,\\
0 & \mbox{otherwise}. \label{eq:stoc0}
\end{cases}
\end{align}
\endgroup
Notice that $p_{i_n,i_{n+1}}$ are the transition probabilities associated to a random walk on the subgraph $G$. Moreover, let us observe now that, for every $x'_{i_0} \in X$,
\begin{align*}
\mu_0 \left(\{x'_{i_0}\}\right)&=\mathbb{P}\left(X'_0=x'_{i_0}\, | \, (X'_0)\in \mathcal{E}\right)\nonumber\\
&=\mathbb{P}\left(X'_0=x'_{i_0}\, | \, X'_0\in X\right)\nonumber\\
&=
\frac{\nu'(\{x'_{i_0}\})}{\sum_{x'_k\in X}\nu'(\{x'_{k}\})}\nonumber\\
&= \nu(\{x'_{i_0}\}).%
\end{align*}
We can then rewrite \eqref{eq:finite_dim_distr_subgraph} in the following way,
\begin{equation}\label{eq:stoc1}
\mu_n \left(\{x'_{i_0}\}\times \ldots \times \{x'_{i_n}\} \right)=
\nu \left(\{x'_{i_0}\}\right)\prod_{k=1}^n p_{i_{k-1},i_k}.
\end{equation}

Since, from \eqref{eq:stoc0},
\begin{align*}
\mathbb{P}\left( X'_{n+1}\in X \, | \, A_{x'_{i_0},\ldots,x'_{i_{n}}}, B_{n+1} \right)&= \sum_{x'_{i_{n+1}}\in X}\mathbb{P}\left( X'_{n+1}=x'_{i_{n+1}} \, | \, A_{x'_{i_0},\ldots,x'_{i_{n}}}, B_{n+1} \right)\\
&= \sum_{x'_{i_{n+1}}\in X}p_{i_n,i_{n+1}}=1,
\end{align*}
then, it is immediate to check that $\mu_{n+1} \left(\{x'_{i_0}\}\times \ldots \times \{x'_{i_n}\}\times X\right)= \mu_n \left(\{x'_{i_0}\}\times \ldots \times \{x'_{i_n}\}\right)$, and since $X$ is a standard Borel space, then by the Kolmogorov extension theorem it is possible to construct a probability measure $\tilde{\mathbb{P}}$ on the sequence space $\left(X^\N, \left(2^{X}\right)^\N\right)$ such that the coordinate maps $\tilde{X}_n(\omega)=\omega_n$ have the desired distribution (see \cite{Durrett2010}).  Finally, observing that, for every $x'_{i_0},\ldots,x'_{i_n}\in X$,
$$
\tilde{\mathbb{P}}\left(\tilde{X}_0 = x'_{i_0}\right)= \nu \left(\{x'_{i_0}\}\right),
$$
and that
$$
\tilde{\mathbb{P}}\left(\tilde{X}_{0} = x'_{i_{0}},\ldots, \tilde{X}_{n} = x'_{i_{n}}\right)= \mu_n \left(\{x'_{i_0}\}\times \cdots\times \{x'_{i_{n}}\}\right)=\nu \left(\{x'_{i_0}\}\right)p_{i_0,i_1}\cdots p_{i_{n-1},i_n},
$$
by \eqref{eq:stoc1}, then by \cite[Theorem 1.1.1]{Norris1997} the stochastic process $(\tilde{X}_n)_{n\in \mathbb{N}}$ is Markov with initial distribution~$\nu$.}
\end{proof}
We call $(\tilde{X}_n)_{n\in \N}$ a \emph{$G$-conditioned random walk}. As a final remark, when $G$ and $G'$ share the same node-set $X$, then $\operatorname{id}(\cdot)$ defines the subgraph $\operatorname{id}(G)\subseteq G'$, but $G$ and $\operatorname{id}(G)$ are not necessarily isomorphic. In particular, $G$ is not necessarily a subgraph of $G'$ {since, following the definition given in Subsection \ref{ssec:subgraph} and Definition \ref{def:phi-subgraph}, $w(x_i,x_j)$ can differ from $w_{\operatorname{id}}(x'_i,x'_j)$, where $x'_i=\operatorname{id}(x_i), x'_j=\operatorname{id}(x_j)$}. On the other hand, $G$ and $\operatorname{id}(G)$ have the same edge-set, and therefore, with abuse of notation, we will write $G$-conditioned random walks on $G'$ {instead of $\operatorname{id(G)}$-conditioned random walks on $G'$}, those random walks on $G'$ that move on the edges of $\operatorname{id}(G)$.

The following corollary is an immediate consequence of Proposition~\ref{prop:conditioned_random_walk} and Definition~\ref{def:SE}
\begin{corollary}\label{cor:restricted_random_walks}
Let $G$ and $G'$ be graphs on the same node-set $X$, and let $E\subset E'$, where $E$ and $E'$ are the edge-set of $G$ and $G'$, respectively. Then $G'$ is compatible with $G$ if and only if every $G$-conditioned random walk $(\tilde{X}_n)_{n\in \N}$ on $G'$ is stochastically equivalent in the wide sense to a random walk on $G$. 
\end{corollary}
\begin{proof}
If $G'$ is compatible with $G$, then $G$ and $\operatorname{id}(G)$ are stochastically equivalent, and therefore their transition probabilities are identical. Therefore, by definition, any $G$-conditioned random walk $(\tilde{X}_n)_{n\in \N}$ on $G'$ with initial distribution $\nu$ is stochastically equivalent in the wide sense to the random walk on $G$ with the same initial distribution $\nu$, because their finite dimensional distributions are identical. Vice-versa, if every $G$-conditioned random walk $(\tilde{X}_n)_{n\in \N}$ on $G'$ is stochastically equivalent in the wide sense to a random walk on $G$, then their finite dimensional distributions are identical. This means that the transition probabilities of $G$ and $\operatorname{id}(G)$ are identical too, and we conclude that $G$ and $\operatorname{id}(G)$ are stochastically equivalent{, that is, $G'$ is compatible with $G$}.
\end{proof}

\subsection{Embedding}\label{sec:embedding}
As we mentioned, there is another way to look at the compatibility between graphs in Definition \ref{def:SE}, that is purely analytic. To make it clear, let us fix a discrete-time random walk $(X_n)_{n\geq 0}$ on $G$ with initial distribution $\nu$, and let us recall that $\hat{\Delta}= \operatorname{id} - P$. Thanks to the Chapman-Kolmogorov equations (see \cite[Theorem 1.1.3]{Norris1997}), then writing $\hat{\Delta}^T$ and $P^T$ for the transpose of $\hat{\Delta}$ and $P$, respectively, by basic manipulation we have that $u(n,\cdot):= P^T(n)[\nu]$ is the unique solution of the following diffusion equation
\begin{equation*}
\begin{cases}
u(n+1,x_i) - u(n,x_i) + \hat{\Delta}^T[u](n,x_i) =0, & (n,x_i) \in \N\times X,\\
u(0,x_i)= \nu.
\end{cases}
\end{equation*}  
The solution $u$ describes the initial probability distribution $\nu$ evolving in time, that is
\begin{equation*}
 u(n,x_i) = \left(P^T(n)[\nu]\right)_i = \mathbb{P}\left(X_n = i \, | \, X_0\sim \nu \right).
\end{equation*}
The same can be said for the continuous-time random walk, replacing the above discrete-time diffusion equation with
\begin{equation*}
\begin{cases}
\partial_tu(t,x_i) + \hat{\Delta}[u](t,x_i) =0, & (t,x_i) \in \R_+\times X,\\
u(0,x_i)= \nu,
\end{cases}
\end{equation*}
with solution $u(t,\cdot)= \textrm{e}^{-t\hat{\Delta}}[\nu(\cdot)]$. The dynamics of a random walk is encoded into the graph Laplacian $\hat{\Delta}$. Then it will not be a big surprise to discover that the stochastic equivalence is the other side of a purely analytic/geometric property. This brings us to the following definition.
\begin{definition}[Embedding]\label{def:embedding&dynamics}
Given two graphs $G=(X,w,\mu)$ and $G'=(X', w',\mu')$, let $\phi : X \to \phi(X)\subseteq X'$ be a bijection onto its image $\phi(X)$ such that satisfies \eqref{hp:embedding_map1}. We say that $G$ can be \emph{embedded} into $G'$ through $\phi$, and we write $G \hookrightarrow_\phi G'$, if 
\begin{subequations}
\begin{equation}\label{hp:embedding_map3}
\tag{\textbf{E2}}\Delta[u] = \Delta_{\phi(G)}[u^*],
\end{equation}
\end{subequations}
where $\Delta_{\phi(G)}$ is the graph Laplacian associated to the $\phi$-induced subgraph $\phi(G)$, defined in \ref{def:phi-subgraph}.
\end{definition}
Let us observe that condition \eqref{hp:embedding_map1} can be seen as a kind of topological embedding, while asking the validity of both \eqref{hp:embedding_map1} and \eqref{hp:embedding_map3} is a kind of metric embedding. Indeed, in some sense the above definition imitates the isometric embedding between Riemannian manifolds, see \cite{ONeill}. Before proceeding further, let us make the following remarks:
\begin{itemize}
	\item when $\mu'=\deg'$ then $\mu_\phi= \deg'_{|(\phi(X),w_\phi)}=:\deg_{\phi}$;
	\item given $G=(X,w,\boldsymbol{1})$ and the normalized path or fractional graph $\hat{G}_\alpha=(X,w_\alpha,\deg), \hat{G}^\alpha=(X,w^\alpha,\deg)$, then the identity map $\operatorname{id} : X \to X$ clearly satisfies \eqref{hp:embedding_map1} and $u^*=u$. In general, $\hat{G}$ is not a subgraph of $\hat{G}_\alpha$ or $\hat{G}^\alpha$, but we are entitled to ask whether $\hat{G}\hookrightarrow_{{\operatorname{id}}} \hat{G}_\alpha$ or $\hat{G}\hookrightarrow_{{\operatorname{id}}} \hat{G}^\alpha$.
\end{itemize}
The next result provides an effective computational way to determine whether there can be an embedding.
\begin{theorem}\label{thm:embedding}
	Let $\hat{G}=(X,w,\deg)$ and $\hat{G}'=(X',w',\deg')$ be normalized graphs such that there exists a bijection $\phi : X \to \phi(X)\subseteq X'$. Then $\hat{G}\hookrightarrow_\phi \hat{G}'$ if and only if 
\begin{equation}\label{thm:embedding_eq1.2}
		\frac{w(x_i,x_j)}{\deg(x_i)} = \frac{w_\phi(\phi(x_i),\phi(x_j))}{\deg_{\phi}(\phi(x_i))} \quad \mbox{for all } x_i,x_j \in X.
\end{equation}
	In particular, condition \eqref{thm:embedding_eq1.2} is equivalent to
	\begin{equation}\tag{4.5'}\label{thm:embedding_eq1.2'}
	\frac{w(x_i,x_j)}{w(x_i,x_k)} = \frac{w'(\phi(x_i),\phi(x_j))}{w'(\phi(x_i),\phi(x_k))} \qquad \forall\, x_i,x_j,x_k\in X \mbox{ such that } x_j\sim x_i, x_k\sim x_i \, \mbox{in }\hat{G}.
	\end{equation}
\end{theorem}	
\begin{proof} 
Because of the linearity of $\hat{\Delta}$, it is enough to prove \eqref{hp:embedding_map3} for every $u \in C(X)$ of the form $u = \mathds{1}_{\{x_i\}}$. First, let us assume the validity of \eqref{hp:embedding_map1}, and without affecting the generality of the proof and for the sake of notation simplicity, we will suppose hereafter that $X\subseteq X'$ and $\phi=\operatorname{id}$, with $\operatorname{id}:X\to X$ the identity map. Fix $x_i \in X$ and let $\{x_{j_k}\}_{k\geq 1}$ be all the nodes in $X$ such that $x_{j_k}\sim x_i$, or equivalently such that $\{x_i,x_{j_k}\}$ is an edge of $\hat{G}$. If $u=\mathds{1}_{\{x_i\}}$, then
	$$
	\hat{\Delta}\left[\mathds{1}_{\{x_i\}}\right](x_i)= 1= \hat{\Delta}_{\operatorname{id}(G)}\left[\mathds{1}_{\{x_i\}}\right](x_i).
	$$
If instead $u= \mathds{1}_{\{x_{j_1}\}}$, then 
	$$
	\hat{\Delta}\left[\mathds{1}_{\{x_{j_1}\}}\right](x_i) = - \frac{w(x_i,x_{j_1})}{\deg(x_i)} = - \frac{w(x_i,x_{j_1})}{\sum_{k\geq 1}w(x_i,x_{j_k})}
	$$
	and 
	$$
	\hat{\Delta}_{\operatorname{id}(G)}\left[\mathds{1}_{\{x_{j_1}\}}\right](x_i)=  - \frac{w_{{\operatorname{id}}}(x_i,x_{j_1})}{\deg_{{\operatorname{id}}}(x_i)}= - \frac{w'(x_i,x_{j_1})}{\sum_{k\geq 1}w'(x_i,x_{j_k})}.
	$$
Imposing that $\hat{\Delta}\left[\mathds{1}_{\{x_{j_1}\}}\right](x_i)=\hat{\Delta}_{\operatorname{id}(G)}\left[\mathds{1}_{\{x_{j_1}\}}\right](x_i)$ gives
\begin{equation*}%
\frac{w(x_i,x_{j_1})}{\deg(x_i)}=	\frac{w(x_i,x_{j_1})}{\sum_{k\geq1}w(x_i,x_{j_k})}=\frac{w'(x_i,x_{j_1})}{\sum_{k\geq1}w'(x_i,x_{j_k})}= \frac{w_{{\operatorname{id}}}(\operatorname{id}(x_i),\operatorname{id}(x_{j_1}))}{\deg_{{\operatorname{id}}}(\operatorname{id}(x_i))},
\end{equation*}
	and letting $x_i,\mathds{1}_{\{x_{j_1}\}}$ vary on all over the nodes in $X$ we get \eqref{thm:embedding_eq1.2}. Let us observe that if there exists only one node $x_{j_1}$ such that $x_i\sim x_{j_1}$, then both equations \eqref{thm:embedding_eq1.2} and \eqref{thm:embedding_eq1.2'} are satisfied. Supposing instead that $\#\{x_{j_k}\, : \, x_{j_k}\sim x_i\}=m\geq 2$, if we write 
	$$
	r_{k}:= \frac{w(x_i,x_{j_k})}{w(x_i,x_{j_1})} , \quad r'_{k}:= \frac{w'(x_i,x_{j_k})}{w'(x_i,x_{j_1})} \quad k\geq 2,
	$$
	then from \eqref{thm:embedding_eq1.2} we get the system of equations
	\begin{equation*}
	\begin{cases}
	\frac{1}{1+\sum_{k\geq 2}r_k} = \frac{1}{1+\sum_{k\geq 2}r'_k},\\
	\frac{r_2}{1+\sum_{k\geq 2}r_k} = \frac{r'_2}{1+\sum_{k\geq 2}r'_k},\\
	\vdots\\
	\frac{r_m}{1+\sum_{k\geq 2}r_k} = \frac{r'_m}{1+\sum_{k\geq 2}r'_k}.
	\end{cases}
	\end{equation*}
	From the first equation we have that
	$$
	\sum_{k\geq2}r_k=\sum_{k\geq2}r'_k,
	$$
	which plugged into the remaining equations provides
	$$
	r_k=r'_k \quad \mbox{for } k\geq 2.
	$$
	From 
	$$
	\frac{r_k}{r_s} = \frac{r'_k}{r'_s} \quad \mbox{for any } k,s\geq 2,
	$$
	we get \eqref{thm:embedding_eq1.2'}. Finally, let us observe that if \eqref{thm:embedding_eq1.2} holds, then $x_i\sim x_j$ if and only if $\phi(x_i)\sim \phi(x_j)$ and \eqref{hp:embedding_map1} is satisfied.
\end{proof}
We have eventually come to the point where we can prove the equivalence between compatibility and embedding.
\begin{theorem}\label{thm:SE&Embedding_equivalence}
	Given two graphs $G$ and $G'$, then they are stochastically equivalent if and only if there exists a bijection $\phi: X \to X'$ such that $\hat{G}\hookrightarrow_\phi \hat{G}'$ and $\hat{G}'\hookrightarrow_{\phi^{-1}} \hat{G}$. In particular, $G^\alpha$ and $G_\alpha$ are compatible with $G$ if and only if $\hat{G}\hookrightarrow_{\operatorname{id}} \hat{G}^\alpha$ and $\hat{G}\hookrightarrow_{\operatorname{id}} \hat{G}_\alpha$, respectively.
\end{theorem}
\begin{proof}
	Suppose that $G$ and $G'$ are stochastically equivalent. Clearly, the bijection $\psi: I \to I'$ between the index sets translates into a bijection $\phi$ between the node-sets $X, X'$. With abuse of notation, we set $\psi=\phi$. From \eqref{eq:SE} and the definition of transition probabilities, we have that
	\begin{equation}\label{eq:2}
	\frac{w(x_i,x_j)}{\deg(x_i)}=p_{i,j} = p'_{\phi(i),\phi(j)} = \frac{w'(\phi(x_i),\phi(x_j))}{\deg'(\phi(x_i))}.
	\end{equation}
	Therefore, $w(x_i,x_j)>0$ if and only if $w'(\phi(x_i),\phi(x_j))>0$, that is, $x_i\sim x_j$ in $G$ if and only if $\phi(x_i)\sim \phi(x_j)$ in $G'$. This means that $w'(\phi(x_i),\phi(x_j))=w_\phi(\phi(x_i),\phi(x_j))$ and $\deg'(\phi(x_i))=\deg_{\phi}(\phi(x_i))$, as in Definition \ref{def:phi-subgraph}, and then $\hat{G}\hookrightarrow_\phi \hat{G}'$ by \eqref{thm:embedding_eq1.2} of Theorem \ref{thm:embedding}. Trivially, it holds the converse as well, namely, $\hat{G}'\hookrightarrow_{\phi^{-1}} \hat{G}$. For the other way around, if it holds that $\hat{G}$ and $\hat{G}'$ can be embedded into each other by a map $\phi$ and its inverse $\phi^{-1}$, then it is immediate to check that $\phi$ induces a bijection $\psi$ between the node-sets $X,X'$, and that $x_i\sim x_j$ if and only if $\phi(x_i)\sim \phi(x_j)$. Then, again by equation \eqref{thm:embedding_eq1.2}, equation \eqref{eq:2} and the definition of transition probabilities, we get \eqref{eq:SE}. The last part of the theorem is just a particular case, since the map $\operatorname{id}(\cdot)$ is trivially a bijection between the node-sets of $G$ and $G^\alpha, G_\alpha$.  
\end{proof}

\subsection{Examples and Counterexamples}\label{ssec:counterexamples}
Here we present some examples where $G^\alpha, G_\alpha$ are compatible with $G$, and vice-versa, counterexamples where they are not. 
\begin{lemma}\label{lem:embedding_path-graph}
Given a graph $G=(X,w,\boldsymbol{1})$, then the path graph $G_\alpha$, generated by the function $\kappa_{\alpha,n}$ defined in \eqref{eq:path-graph}, is compatible with $G$ if and only if 
\begin{equation}\label{eq:001}
\frac{\kappa_{\alpha,1}(x_i,x_j)}{\kappa_{\alpha,1}(x_i,x_k)}=1 \qquad \forall\, x_i,x_j,x_k \in X \mbox{ such that } x_j\sim x_i, x_k \sim x_i \mbox{ in } G.
\end{equation}
In particular, if $G$ is unweighted then the path graphs generated by the Mellin and Laplace transform functions are compatible with $G$ for every $\alpha >0$.
\end{lemma}
\begin{proof}
It is immediate from Theorem~\ref{thm:SE&Embedding_equivalence} and equation \eqref{thm:embedding_eq1.2'}, and the definition of $\kappa_{\alpha,n}$ in \eqref{eq:path-graph}.
\end{proof}

\begin{lemma}\label{lem:cycle}
Let $G=(X,w,\boldsymbol{1})$ be a cycle of dimension $n$, that is, $X=\{x_i \, : \, i=0,\ldots,n-1\}$ and
$$
w(x_i,x_j)=\begin{cases}
1 & \mbox{if } (i - j)\equiv \pm 1 \mod{n},\\
0 & \mbox{otherwise},
\end{cases}
$$
Then $G^\alpha$ is compatible with $G$ for every $\alpha \in (0,1]$.
\end{lemma}
\begin{proof}
In virtue of Theorem \ref{thm:SE&Embedding_equivalence}, we will prove that $\hat{G} \hookrightarrow_{\operatorname{id}} \hat{G}^\alpha$. It is clear that the degree function is uniform, that is, $\deg(x_i)\equiv 2$ for every $x_i \in X$. Namely, $G$ is a finite regular graph and by the matrix notation \eqref{def:graph-Laplacian_matrix_form2}, it holds then that
\begin{equation*}
\Delta = 2\cdot\operatorname{id} - A(G).
\end{equation*} 
Therefore, $\Delta^\alpha = U\left(2\cdot\operatorname{id}-\Lambda\right)^\alpha U^T$, where $U\Lambda U^T$ is the spectral decomposition of the adjacency matrix $A(G)$, which is circulant and symmetric. By standard theory (see \cite{Gray2006}), it holds that
\begin{equation*}
(\Lambda)_{k,k}= 2\cos\left(\frac{2\pi k}{n}\right), \qquad (U)_{i,j}= \frac{1}{\sqrt{n}}\textrm{e}^{-2\pi\textrm{i}\frac{ij}{n}}, \qquad \mbox{for } i,j,k =0,\ldots,n-1.
\end{equation*}
Therefore, since $w^\alpha(x_i,x_j) = (\Delta^\alpha)_{i,j}$, then
\begin{align*}
w^\alpha(x_i,x_j)&=\frac{1}{n}\sum_{k=0}^{n-1}\left(2- 2\cos\left(\frac{2\pi k}{n}\right)\right)^\alpha\textrm{e}^{-2\pi\textrm{i}\frac{ki}{n}}\textrm{e}^{2\pi\textrm{i}\frac{kj}{n}}\\
&= \frac{1}{n}\sum_{k=1}^{n-1}\left(4\sin^2\left(\frac{\pi k}{n}\right)\right)^\alpha\textrm{e}^{-2\pi\textrm{i}\frac{k(i-j)}{n}}\\
&= \frac{1}{n}\sum_{k=1}^{n-1}\left(4\sin^2\left(\frac{\pi k}{n}\right)\right)^\alpha\left[\cos\left(\frac{2\pi k(i-j)}{n}\right)+ \textrm{i}\sin\left(\frac{2\pi k(i-j)}{n}\right) \right]\\
&= \frac{1}{n}\sum_{k=1}^{n-1}\left(4\sin^2\left(\frac{\pi k}{n}\right)\right)^\alpha\cos\left(\frac{2\pi k(i-j)}{n}\right).
\end{align*}
Because $x_i\sim x_j$ in $G$ if and only if $(i - j)\equiv \pm 1 \mod{n}$, then it is immediate to check now that
\begin{equation*}
\frac{w^\alpha(x_i,x_j)}{w^\alpha(x_i,x_{k})} = 1 \qquad \forall \alpha \in (0,1],\, \forall x_i,x_j,x_k \in X \mbox{ such that } x_i\sim x_j,\, x_i \sim x_k \, \mbox{in } \hat{G}.
\end{equation*}
By Theorem \ref{thm:embedding} we conclude that $\hat{G} \hookrightarrow_{\operatorname{id}} \hat{G}^\alpha$.
\end{proof}
The above lemma and the next counterexamples can be used to produce other embeddings or counterexamples, respectively, by creating new graphs just from the Kronecker sums of their adjacency matrices. For instance, the cubic lattice with periodic boundary conditions of dimension $d$ is obtained by $d$ Kronecker sums of the $1d$ cycle, see \cite{Michelitsch2017}.

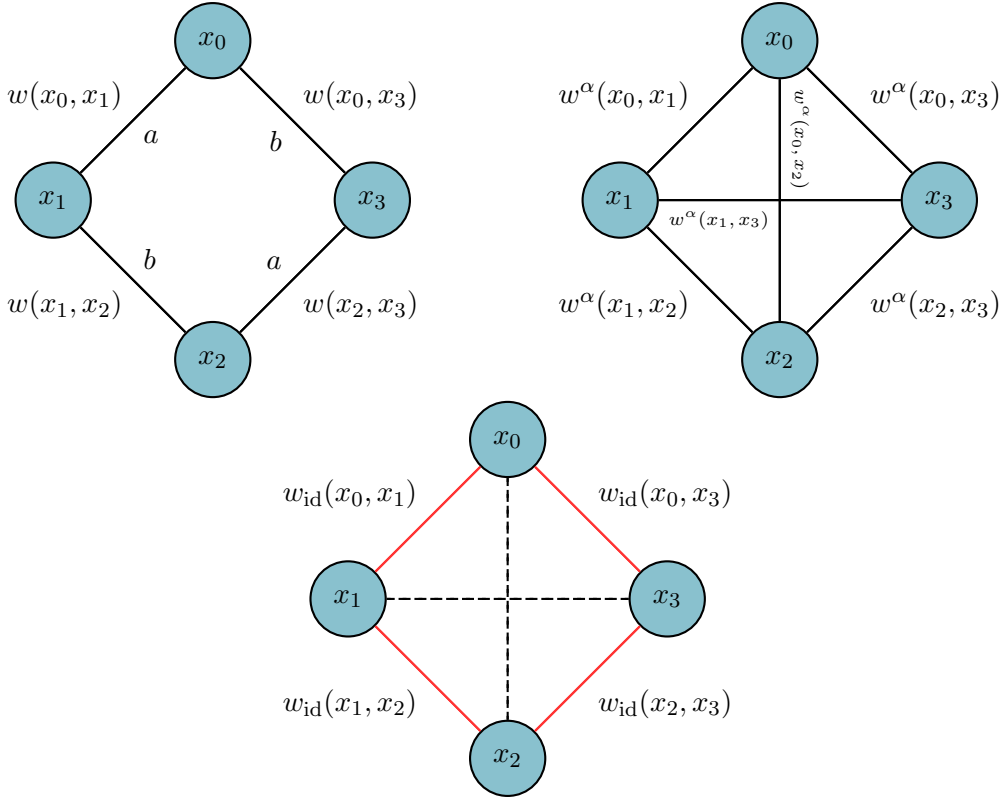
\begin{figure}
	\begin{center}
		\begin{minipage}{.45\textwidth}
			\begin{tikzpicture}[-latex, auto,node distance =3 cm ,on grid ,
			thick ,
			state2/.style ={ circle ,top color =white , bottom color = gray ,
				draw,black , text=black , minimum width =1 cm},
			state3/.style ={ circle ,top color =white , bottom color = white ,
				draw,white , text=white , minimum width =1 cm},
			state/.style={circle,top color = palecerulean , bottom color = palecerulean,
				draw,black , text=black , minimum width =1 cm}]
			
			\node[state] (E) [] {$x_{0}$};
			\node[state] (F) [below left=of E] {$x_{1}$};
			\node[state] (G) [below right=of F] {$x_2$};
			\node[state] (H) [above right=of G] {$x_{3}$};

			\path (G) edge [black,-] node[] {$a$} (H);
			\path (H) edge [black,-] node[] {$w(x_2,x_3)$} (G);
			\path (G) edge [black,-] node[] {$w(x_1,x_2)$} (F);
			\path (F) edge [black,-] node[] {$b$} (G);
			\path (E) edge [black,-] node[] {$a$} (F);
			\path (F) edge [black,-] node[] {$w(x_0,x_1)$} (E);
			\path (H) edge [black,-] node[] {$b$} (E);
			\path (E) edge [black,-] node[] {$w(x_0,x_3)$} (H);
			\end{tikzpicture}
		\end{minipage}
		\begin{minipage}{.45\textwidth}
			\begin{tikzpicture}[-latex, auto,node distance =3 cm ,on grid ,
			thick ,
			state2/.style ={ circle ,top color =white , bottom color = gray ,
				draw,black , text=black , minimum width =1 cm},
			state3/.style ={ circle ,top color =white , bottom color = white ,
				draw,white , text=white , minimum width =1 cm},
			state/.style={circle ,top color =palecerulean , bottom color =palecerulean,
				draw,black , text=black , minimum width =1 cm}]
			
			\node[state] (E) [] {$x_{0}$};
			\node[state] (F) [below left=of E] {$x_{1}$};
			\node[state] (G) [below right=of F] {$x_2$};
			\node[state] (H) [above right=of G] {$x_{3}$};

			\path (G) edge [black,-] node[] {} (H);
			\path (H) edge [black,-] node[] {$w^\alpha(x_2,x_3)$} (G);
			\path (G) edge [black,-] node[] {$w^\alpha(x_1,x_2)$} (F);
			\path (F) edge [black,-] node[] {} (G);
			\path (E) edge [black,-] node[] {} (F);
			\path (F) edge [black,-] node[] {$w^\alpha(x_0,x_1)$} (E);
			\path (H) edge [black,-] node[] {} (E);
			\path (E) edge [black,-] node[] {$w^\alpha(x_0,x_3)$} (H);
			
			\path (E) edge [black,-] node[sloped, above left] {\tiny{$w^\alpha(x_0,x_2)$}} (G);
			\path (G) edge [black,-] node[] {} (E);
			\path (F) edge [black,-] node[below left] {\tiny{$w^\alpha(x_1,x_3)$}} (H);
			\path (H) edge [black,-] node[] {} (F);
			\end{tikzpicture}
		\end{minipage}
		\begin{minipage}{.45\textwidth}
			\begin{tikzpicture}[-latex, auto,node distance =3 cm ,on grid ,
			thick ,
			state2/.style ={ circle ,top color =white , bottom color = gray ,
				draw,black , text=black , minimum width =1 cm},
			state3/.style ={ circle ,top color =white , bottom color = white ,
				draw,white , text=white , minimum width =1 cm},
			state/.style={circle ,top color =palecerulean , bottom color = palecerulean,
				draw,black , text=black , minimum width =1 cm}]
			
			\node[state] (E) [] {$x_{0}$};
			\node[state] (F) [below left=of E] {$x_{1}$};
			\node[state] (G) [below right=of F] {$x_2$};
			\node[state] (H) [above right=of G] {$x_{3}$};

			\path (G) edge [color=red!80,-] node[] {} (H);
			\path (H) edge [color=red!80,-] node[] {\textcolor{black}{$w_{\operatorname{id}}(x_2,x_3)$}} (G);
			\path (G) edge [color=red!80,-] node[] {\textcolor{black}{$w_{\operatorname{id}}(x_1,x_2)$}} (F);
			\path (F) edge [color=red!80,-] node[] {} (G);
			\path (E) edge [color=red!80,-] node[] {} (F);
			\path (F) edge [color=red!80,-] node[] {\textcolor{black}{$w_{\operatorname{id}}(x_0,x_1)$}} (E);
			\path (H) edge [color=red!80,-] node[] {} (E);
			\path (E) edge [color=red!80,-] node[] {\textcolor{black}{$w_{\operatorname{id}}(x_0,x_3)$}} (H);
			
			\path[dashed] (E) edge [black,-] node[sloped, above left] {} (G);
			\path[dashed] (G) edge [black,-] node[] {} (E);
			\path[dashed] (F) edge [black,-] node[below left] {} (H);
			\path[dashed] (H) edge [black,-] node[] {} (F);
			\end{tikzpicture}
		\end{minipage}
	\end{center}
	\caption{Visual representation of the graph $G$ in Counterexample \ref{example:1} (up-left), of the graph $G^\alpha$ (up-right) characterized by the edge-weights $w^\alpha$ induced by $\Delta^\alpha$, and finally of the subgraph $\operatorname{id}(G)\subset G^\alpha$ (bottom), where the common edges shared by $G^\alpha$ and $\operatorname{id}(G)$ are depicted in red. In particular, $w_{\operatorname{id}}(x_i,x_j)=w^\alpha(x_i,x_j)$.}\label{fig:ex:1}
\end{figure}
\begin{cexample}\label{example:1}
Let us consider the graph $G$ in Figure \ref{fig:ex:1} (up-left):
\begin{itemize}
	\item $X=\{x_0, x_1, x_2, x_3\}$;
	\item $w(x_0,x_2)=w(x_1,x_3)=0$, $w(x_0,x_1)=w(x_2,x_3)=a$, $w(x_0,x_3)=w(x_1,x_2)=b$;
	\item $a+b=1$;
	\item $\mu(x_i)=\deg(x_i) = 1$ for every $x_i\in X$.
\end{itemize}	
That is, $G$ is a cycle as in Lemma \ref{lem:cycle}, but weighted. Since in this example we choose the weights such that the counting measure and the degree measure agree, then $G=\hat{G}$ and it holds that
$$
\Delta = \hat{\Delta}= \begin{bmatrix}
1 & \shortminus a & 0 & \shortminus b \\
\shortminus a & 1 & \shortminus b & 0 \\
0 & \shortminus b & 1 & \shortminus a \\
\shortminus b & 0 & \shortminus a & 1
\end{bmatrix}.
$$
By easy computations we have that	
\begin{equation}\label{eq:1}
\Delta^\alpha = \begin{bmatrix} \frac{(2a)^\alpha + (2b)^\alpha + 2^\alpha }{4} &
\frac{-(2a)^\alpha + (2b)^\alpha - 2^\alpha }{4}&
\frac{-(2a)^\alpha - (2b)^\alpha + 2^\alpha }{4}&
\frac{(2a)^\alpha - (2b)^\alpha - 2^\alpha }{4}\\[0.3cm]
\frac{-(2a)^\alpha + (2b)^\alpha - 2^\alpha }{4}&
\frac{(2a)^\alpha + (2b)^\alpha + 2^\alpha }{4} &
\frac{(2a)^\alpha - (2b)^\alpha - 2^\alpha }{4}&
\frac{-(2a)^\alpha - (2b)^\alpha + 2^\alpha }{4}\\[0.3cm]
\frac{-(2a)^\alpha - (2b)^\alpha + 2^\alpha }{4}&
\frac{(2a)^\alpha - (2b)^\alpha - 2^\alpha }{4}&
\frac{(2a)^\alpha + (2b)^\alpha + 2^\alpha }{4}&
\frac{-(2a)^\alpha + (2b)^\alpha - 2^\alpha }{4}\\[0.3cm]
\frac{(2a)^\alpha - (2b)^\alpha - 2^\alpha }{4}&
\frac{-(2a)^\alpha - (2b)^\alpha + 2^\alpha }{4}&
\frac{-(2a)^\alpha + (2b)^\alpha - 2^\alpha }{4}&
\frac{(2a)^\alpha + (2b)^\alpha + 2^\alpha }{4}
\end{bmatrix},
\end{equation}	
from which we get the new edge-weights $w^\alpha(x_i,x_j):=\left(\Delta^\alpha\right)_{i,j}$ associated to the graph $G^\alpha=(X,w^\alpha,\boldsymbol{1})$, see Figure \ref{fig:ex:1} (up-right).  The subgraph $\operatorname{id}(G)\subset G^\alpha$ (see Definition \ref{def:embedding&dynamics} and Figure \ref{fig:ex:1}, bottom), is such that $\operatorname{id}(G)=(X,w_{{\operatorname{id}}}, \boldsymbol{1})$:
\begin{itemize}
	\item $X=\{x_0, x_1, x_2, x_3\}$;
	\item $w_{{\operatorname{id}}}(x_i,x_j)= w^\alpha(x_i,x_j)$ if and only if $\{x_i,x_j\}$ is an edge of $G$, otherwise is zero.
\end{itemize}
Consider now the normalizations $\hat{G}$, $\hat{G}^\alpha$ and $\operatorname{id}\left(\hat{G}\right)\subset \hat{G}^\alpha$. From Theorem \ref{thm:embedding} we have that $\hat{\Delta} = \hat{\Delta}_{\operatorname{id}(\hat{G})}$, that is, $\hat{G}\hookrightarrow_{{\operatorname{id}}} \hat{G}^\alpha$ if and only if 
$$
\begin{cases}
\frac{a}{b}=\frac{w(x_0,x_1)}{w(x_0,x_3)} =   \frac{w^\alpha(x_0,x_1)}{w^\alpha(x_0,x_3)}, \\ \frac{a}{b}=\frac{w(x_1,x_0)}{w(x_1,x_2)} =   \frac{w^\alpha(x_1,x_0)}{w^\alpha(x_1,x_2)},\\
\frac{b}{a}=\frac{w(x_2,x_1)}{w(x_2,x_3)} =   \frac{w^\alpha(x_2,x_1)}{w^\alpha(x_2,x_3)}, \\
\frac{b}{a}=\frac{w(x_3,x_2)}{w(x_3,x_0)} =   \frac{w^\alpha(x_3,x_2)}{w^\alpha(x_3,x_0)}.
\end{cases}
$$
By direct computations, it can be proved that, for example,
$$
\frac{a}{b} = \frac{w(x_0,x_1)}{w(x_0,x_3)} =   \frac{w^\alpha(x_0,x_1)}{w^\alpha(x_0,x_3)},
$$	
if and only if $a=b=1/2$ (for any $\alpha\in(0,1)$), in agreement with Lemma \ref{lem:cycle}. So, in the case of a weighted cycle ($a\neq b$), $\hat{G}$ can not be embedded in $\hat{G}^\alpha$, i.e., $G^\alpha$ is not compatible with $G$.
\end{cexample}

\begin{cexample}\label{example:2}
We consider here the graph $G$ in Figure \ref{fig:ex:2} (up-left), which is a modification of the graph in Counterexample \ref{example:1} equipped with the counting measure: the difference is that now $G$ is simple with $w(x_0,x_3)=0$ and  $w(x_0,x_1)=w(x_1,x_2)=w(x_2,x_3)=1$.
It holds that
	$$
	\Delta =\begin{bmatrix}
	1 & \shortminus 1 & 0 & 0 \\
	\shortminus 1 & 2 & \shortminus 1 & 0 \\
	0 & \shortminus 1 & 2 & \shortminus 1 \\
	0 & 0 & \shortminus 1 & 1
	\end{bmatrix}, \qquad 
	\hat{\Delta} =\begin{bmatrix}
	1 & \shortminus 1 & 0 & 0 \\
	\shortminus 0.5 & 1 & \shortminus 0.5 & 0 \\
	0 & \shortminus 0.5 & 1 & \shortminus 0.5 \\
	0 & 0 & \shortminus 1 & 1
	\end{bmatrix},
	$$
where clearly $\hat{\Delta}$ is the graph Laplacian associated to the normalization $\hat{G}$ of $G$. Noticing that
$$
\Delta = \begin{bmatrix}
2 &0 & 0 & 0 \\
0 & 2 & 0 & 0 \\
0 & 0 & 2 & 0 \\
0 & 0 & 0 & 2
\end{bmatrix}-
\begin{bmatrix}
1 &  1 & 0 & 0 \\
 1 & 0 &  1 & 0 \\
0 &  1 & 0 &  1 \\
0 & 0 &  1 & 1
\end{bmatrix},
$$
then for any fixed $\alpha \in (0,1]$, it is possible to explicitly express the weights $w^\alpha(x_i,x_j)$ of the fractional graph $G^\alpha$:
$$
w^\alpha(x_i,x_j)=\left(\Delta^\alpha\right)_{i,j}= \frac{1}{2}\sum_{k=1}^3 \left[4\sin^2\left(\frac{\pi k}{8}\right)\right]^\alpha \cos\left(\frac{\pi k(2i+1)}{8}\right)\cos\left(\frac{\pi k(2j+1)}{8}\right),
$$
see \cite[Equation (22)]{Bozzo1995}. By direct computation, it is now immediate to check that
	$$
1 = \frac{w(x_1,x_0)}{w(x_1,x_2)} \neq  \frac{w^\alpha(x_1,x_0)}{w^\alpha(x_1,x_2)}, \quad 1 = \frac{w(x_2,x_1)}{w(x_2,x_3)} \neq  \frac{w^\alpha(x_2,x_1)}{w^\alpha(x_2,x_3)}
	$$	
for any $\alpha\neq 1$. Therefore, by Theorem \ref{thm:embedding} we have that $\hat{\Delta}\neq \hat{\Delta}^\alpha_{\operatorname{id}(\hat{G})}$, that is, $\hat{G}$ can not be embedded in $\hat{G}^\alpha$, and therefore $G^\alpha$ is not compatible with $G$.
\begin{figure}
	\begin{center}
		\begin{minipage}{.45\textwidth}
			\begin{tikzpicture}[-latex, auto,node distance =3 cm ,on grid ,
			semithick ,
			state2/.style ={ circle ,top color =white , bottom color = gray ,
				draw,black , text=black , minimum width =1 cm},
			state3/.style ={ circle ,top color =white , bottom color = white ,
				draw,white , text=white , minimum width =1 cm},
			state/.style={circle ,top color =palecerulean , bottom color =palecerulean,
				draw,black , text=black , minimum width =1 cm}]
			
			\node[state] (E) [] {$x_{0}$};
			\node[state] (F) [below left=of E] {$x_{1}$};
			\node[state] (G) [below right=of F] {$x_2$};
			\node[state] (H) [above right=of G] {$x_{3}$};

			\path (G) edge [black,-] node[] {$1$} (H);
			\path (H) edge [black,-] node[] {$w(x_2,x_3)$} (G);
			\path (G) edge [black,-] node[] {$w(x_1,x_2)$} (F);
			\path (F) edge [black,-] node[] {$1$} (G);
			\path (E) edge [black,-] node[] {$1$} (F);
			\path (F) edge [black,-] node[] {$w(x_0,x_1)$} (E);
			\end{tikzpicture}
		\end{minipage}
		\begin{minipage}{.45\textwidth}
			\begin{tikzpicture}[-latex, auto,node distance =3 cm ,on grid ,
			semithick ,
			state2/.style ={ circle ,top color =white , bottom color = gray ,
				draw,black , text=black , minimum width =1 cm},
			state3/.style ={ circle ,top color =white , bottom color = white ,
				draw,white , text=white , minimum width =1 cm},
			state/.style={circle ,top color =palecerulean , bottom color =palecerulean,
				draw,black , text=black , minimum width =1 cm}]
			
			\node[state] (E) [] {$x_{0}$};
			\node[state] (F) [below left=of E] {$x_{1}$};
			\node[state] (G) [below right=of F] {$x_2$};
			\node[state] (H) [above right=of G] {$x_{3}$};

			\path (G) edge [black,-] node[] {} (H);
			\path (H) edge [black,-] node[] {$w^\alpha(x_2,x_3)$} (G);
			\path (G) edge [black,-] node[] {$w^\alpha(x_1,x_2)$} (F);
			\path (F) edge [black,-] node[] {} (G);
			\path (E) edge [black,-] node[] {} (F);
			\path (F) edge [black,-] node[] {$w^\alpha(x_0,x_1)$} (E);
			\path (H) edge [black,-] node[] {} (E);
			\path (E) edge [black,-] node[] {$w^\alpha(x_0,x_3)$} (H);
			
			\path (E) edge [black,-] node[sloped, above left] {\tiny{$w^\alpha(x_0,x_2)$}} (G);
			\path (G) edge [black,-] node[] {} (E);
			\path (F) edge [black,-] node[below left] {\tiny{$w^\alpha(x_1,x_3)$}} (H);
			\path (H) edge [black,-] node[] {} (F);
			\end{tikzpicture}
		\end{minipage}
		\begin{minipage}{.45\textwidth}
			\begin{tikzpicture}[-latex, auto,node distance =3 cm ,on grid ,
			semithick ,
			state2/.style ={ circle ,top color =white , bottom color = gray ,
				draw,black , text=black , minimum width =1 cm},
			state3/.style ={ circle ,top color =white , bottom color = white ,
				draw,white , text=white , minimum width =1 cm},
			state/.style={circle ,top color =palecerulean , bottom color = palecerulean,
				draw,black , text=black , minimum width =1 cm}]
			
			\node[state] (E) [] {$x_{0}$};
			\node[state] (F) [below left=of E] {$x_{1}$};
			\node[state] (G) [below right=of F] {$x_2$};
			\node[state] (H) [above right=of G] {$x_{3}$};

			\path (G) edge [color=red!80,-] node[] {} (H);
			\path (H) edge [color=red!80,-] node[] {\textcolor{black}{$w_{\operatorname{id}}(x_2,x_3)$}} (G);
			\path (G) edge [color=red!80,-] node[] {\textcolor{black}{$w_{\operatorname{id}}(x_1,x_2)$}} (F);
			\path (F) edge [color=red!80,-] node[] {} (G);
			\path (E) edge [color=red!80,-] node[] {} (F);
			\path (F) edge [color=red!80,-] node[] {\textcolor{black}{$w_{\operatorname{id}}(x_0,x_1)$}} (E);
			\path[dashed] (H) edge [black,-] node[] {} (E);
			\path[dashed] (E) edge [black,-] node[] {} (H);
			
			\path[dashed] (E) edge [black,-] node[sloped, above left] {} (G);
			\path[dashed] (G) edge [black,-] node[] {} (E);
			\path[dashed] (F) edge [black,-] node[below left] {} (H);
			\path[dashed] (H) edge [black,-] node[] {} (F);
			\end{tikzpicture}
		\end{minipage}
	\end{center}
	\caption{Visual representation of the graph $G$ in Counterexample \ref{example:2} (up-left), of the graph $G^\alpha$ (up-right) characterized by the edge-weights $w^\alpha$ induced by $\Delta^\alpha$, and finally of the subgraph $\operatorname{id}(G)\subset G^\alpha$ (bottom), where the common edges shared by $G^\alpha$ and $\operatorname{id}(G)$ are depicted in red. In particular, $w_{\operatorname{id}}(x_i,x_j)=w^\alpha(x_i,x_j)$.}\label{fig:ex:2}
\end{figure}
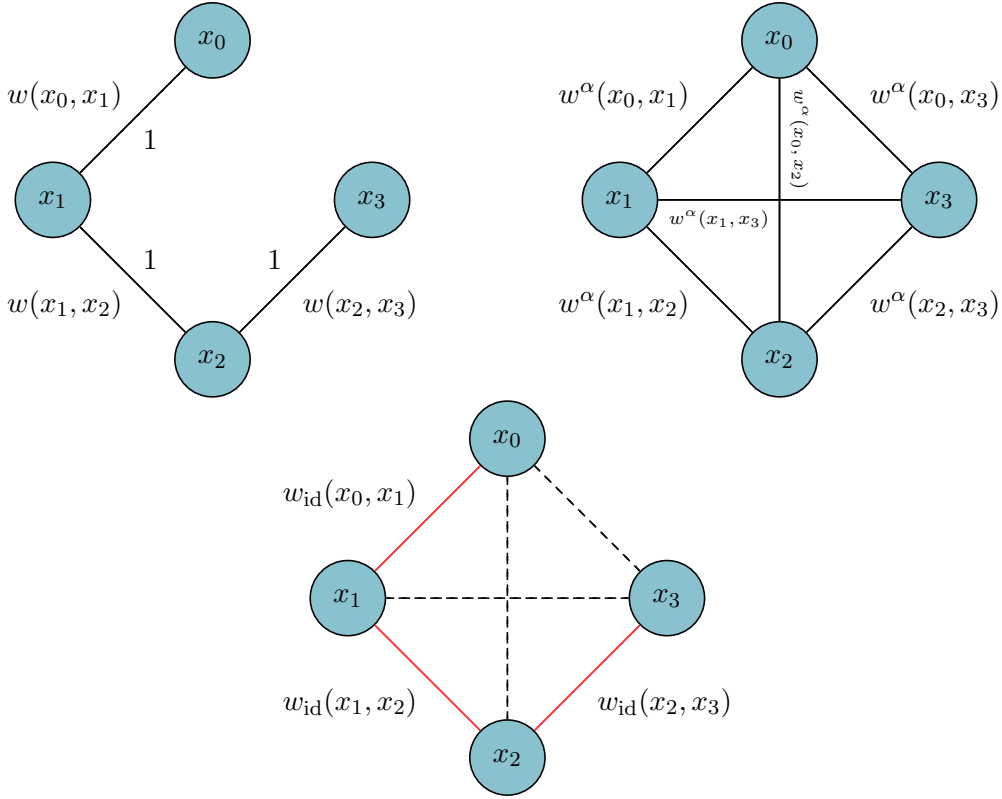
\end{cexample}	

\begin{cexample}\label{cexample:dpathlaplacian}
Consider now again a cycle $G=(X,w,\boldsymbol{1})$ with $X=\{x_0,x_1,x_2, x_3\}$ and $w$ given by
$$
w(x_0,x_1)=1, \quad w(x_1,x_2)=\frac{1}{10}, \quad w(x_2,x_3)=\frac{1}{10}, \quad w(x_3,x_0)=\frac{1}{10}.
$$
Let $d_w$ be the shortest-path distance, that is
$$
d_w(x_i,x_j):= \min_{n\in \N} \left\{ \sum_{k=1}^n w(x_{i_k}, x_{i_{k+1}}) \, : \, x_{i_1}=x_i\sim\ldots \sim x_{i_n}=x_j \right\},
$$
and $h_\alpha(t)=t^{-\alpha}$. It is immediate to check that
$$
\frac{w(x_0,x_1)}{w(x_0,x_3)}=10 \neq \frac{1}{3^\alpha} = \frac{\kappa_{\alpha,1}(x_0,x_1)}{\kappa_{\alpha,1}(x_0,x_3)},
$$
that is, by \eqref{eq:001} of Lemma \ref{lem:embedding_path-graph}, $G_\alpha$ is not compatible with $G$ for any $\alpha >0$.
\end{cexample}

\section{Regularized fractional graph Laplacian}\label{sec:regularized_fractional_graph-Laplacian}
Let $\Delta^\alpha$ be the fractional power of the graph Laplacian described in Subsection~\ref{ssec:fractional_graph_laplacian} for a connected graph $G = (X,w,\boldsymbol{1})$. This means that $\Delta^\alpha$ is the Laplacian of a new complete loop-less graph $G^\alpha=(X,w^\alpha,\boldsymbol{1})$ with weights
\begin{equation}
w^{\alpha}(x_i,x_j) = - (\Delta^\alpha)_{i,j}, \qquad i\neq j.
\end{equation}
We want to produce a regularized version $\prescript{}{r}{G}^\alpha= (X,\prescript{}{r}{w}^\alpha,\boldsymbol{1})$ of $G^\alpha$ for which it holds that the normalization of $G$ can be always embedded in the normalization of $\prescript{}{r}{G}^\alpha$, i.e., $\hat{G} \hookrightarrow_{\operatorname{id}} \prescript{}{r}{\hat{G}}^\alpha$ for $\operatorname{id}$ the identity map. To achieve this we consider the graph $\prescript{}{r}{G}^\alpha$ built on the same nodes $X$ of $G$ and $G^\alpha$, and edge-weight function
\begin{equation}\label{eq:regularized_weight}
\prescript{}{r}{w}^\alpha(x_i,x_j):= \begin{cases}
-\left(\Delta^\alpha\right)_{i,j} & \mbox{if } \{x_i,x_j\} \notin E,\\
w(x_i,x_j)  & \mbox{if } \{x_i,x_j\} \in E.
\end{cases}
\end{equation}
where $E$ is the set of edges with respect to the original weight function, i.e., $E = \{ (x_i,x_j) \,:\; w(x_i,x_j) \neq 0 \}$. Then, the graphs $\prescript{}{r}{\hat{G}}^\alpha$ and $\hat{G}$ satisfy by construction condition~\eqref{thm:embedding_eq1.2'} of Theorem~\ref{thm:embedding}, and therefore $\hat{G} \hookrightarrow_{\operatorname{id}} \prescript{}{r}{\hat{G}}^\alpha$. Equivalently, $\prescript{}{r}{G}^\alpha$ is compatible with $G$.

\begin{proposition}
	Given a weighted connected graph $G = (X,w,\boldsymbol{1})$, the graph $\prescript{}{r}{G}^\alpha = (X,\prescript{}{r}{w}^\alpha,\boldsymbol{1})$ for $\prescript{}{r}{w}^\alpha$ in~\eqref{eq:regularized_weight} is such that $\hat{G} \hookrightarrow_{\operatorname{id}} \prescript{}{r}{\hat{G}}^\alpha$ for $\operatorname{id}$ the identity map, that is, $\prescript{}{r}{G}^\alpha$ is compatible with $G$.
\end{proposition}

\begin{example}\label{example:cycleembedding}
	Let us expand Counterexample \ref{example:1} in this way:
	\begin{itemize}
		\item $X=\{x_i \, : \, i=1,\ldots,10\}$;
		\item the graph is a cycle defined by the walk $x_1 \sim x_2 \sim \ldots \sim x_{10} \sim x_1$;
		\item $w(x_1,x_{10})=10$, and all the other weights are set to $w(x_i,x_{i+1}) = 1$ for $i=1,\ldots,9$.
	\end{itemize}
	Assuming that we will start all the random walks from $x_1$, the dynamics induced by $w$ tells us that we will most likely stay in $x_1$ and $x_{10}$, and that we will rarely visit the central nodes. It is then desirable to introduce the possibility to have jumps between nodes, in order to let our random walker to be able to visit more nodes. We have seen that introducing jumps is the same as introducing new edges between nodes, and we will do it in two ways: by the fractional graph $G^\alpha$ induced by $\Delta^\alpha$ and described in Section~\ref{ssec:fractional_graph_laplacian}, and by the graph $\prescript{}{r}{G}^\alpha$ described in Section \ref{sec:regularized_fractional_graph-Laplacian}. The result for this setting are depicted in Figure~\ref{fig:cycleembedding}.
	\begin{figure}[htbp]
		\centering
		\includegraphics[width=\columnwidth]{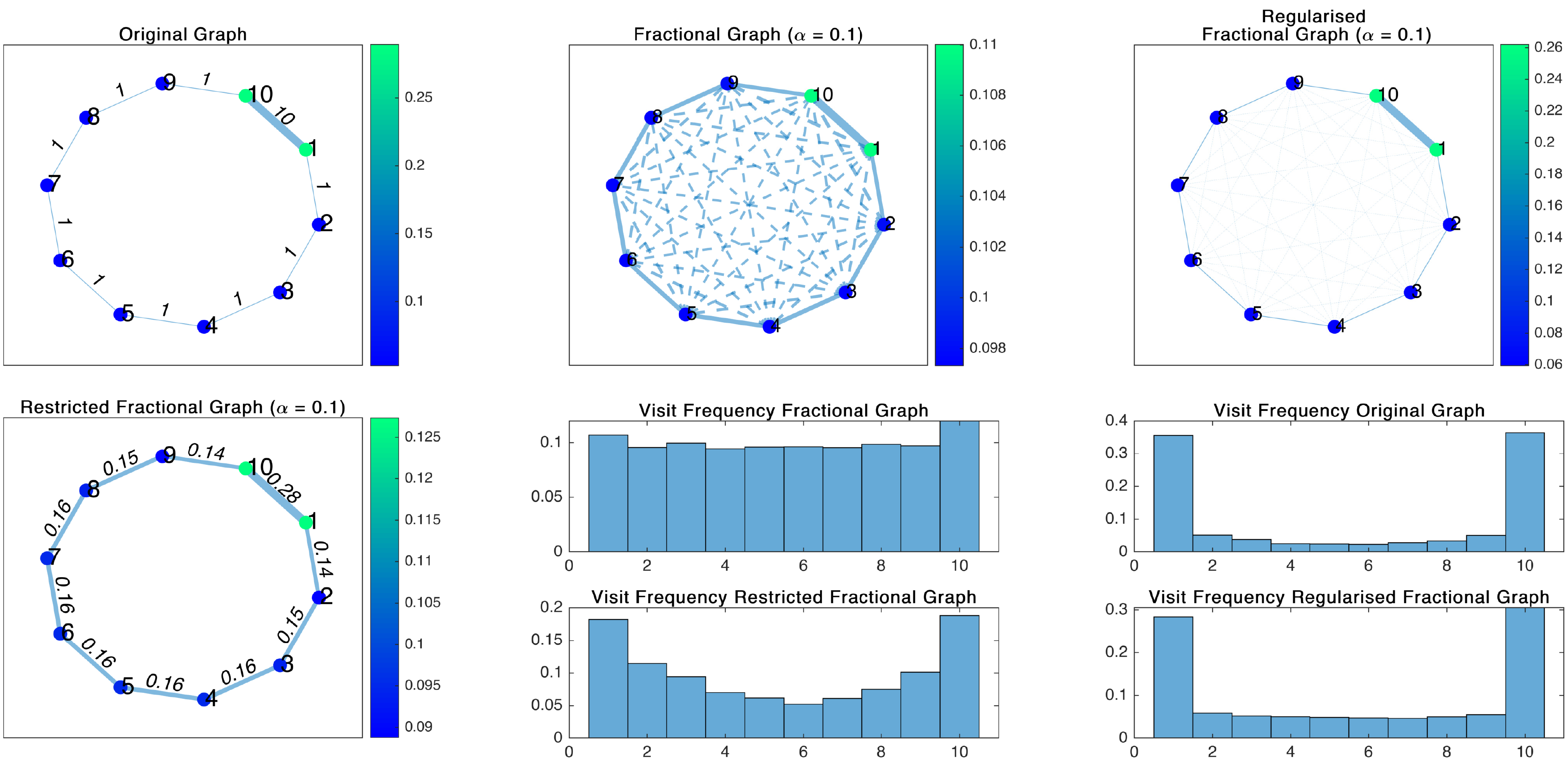}
		\caption{Example~\ref{example:cycleembedding}. The panels on the first line represent the weights of the original graph $G$, of $G^\alpha$, and of $\prescript{}{r}{G}^\alpha$. The width of the edges is proportional to its weight, while the color of the nodes is proportional to the value of the stationary distribution induced by the relative graph Laplacian matrix. On the second line, we report the weights induced on $G$ by $\Delta^\alpha$, and as we can observe their ratio differ from the original ones. The histograms represent the frequency visit to the nodes obtained in 1000 random walks of 20 steps simulated on the graphs $G$, $G^\alpha$, $\prescript{}{r}{G}^\alpha$, and on the graphs $G^\alpha$, $\prescript{}{r}{G}^\alpha$ conditioned to the edges of $G$.}
		\label{fig:cycleembedding}
	\end{figure}
	In the first panel of the second line we have reported the weights for the edges of the original graph $G$ induced by $G^\alpha$, and as we have proven in Counterexample~\ref{example:1} their ratio are different from the original ones; they can be confronted with the ones that are given instead in the first panel of the first line. This is confirmed also by the fact that the frequency histograms for the visit of the nodes of the original graph, and of the fractional Laplacian conditioned on the edges of $G$, in the last panel of the second line, are sensibly different. It is exactly the meaning of Corollary~\ref{cor:restricted_random_walks}: $G^\alpha$ is a complete graph on the same node-set $X$ of $G$, and therefore $E^\alpha \supset E$; then the dynamics on $G^\alpha$ is compatible with the dynamics on $G$ if and only if any random walk on $G^\alpha$ conditioned to move only along the edges of $G$ is scholastically equivalent in the wide sense with a random walk on $G$. It is clear then, from the visit frequencies of the $G$-conditioned random walks on $G^\alpha$ and $\prescript{}{r}{G}^\alpha$, that $G^\alpha$ is not compatible with $G$ while $\prescript{}{r}{G}^\alpha$ is. The second and third panel of the first line represents the graphs $G^\alpha$ and $\prescript{}{r}{G}^\alpha$, respectively. Finally in the second panel of the second line we observe the difference between the regularized and the standard fractional graph. Indeed, the regularized version maintains the predominance of the nodes $x_1$, and $x_{10}$ while still increasing the visit to the other nodes, while $G^\alpha$ makes all the nodes indistinguishable.
\end{example}

The proposed modification of the weight function on $G$ to preserve the original dynamics, maintains also the motivating choice of having long-range diffusion and random walks on the graph.  Indeed, the decay behavior of the entries of the fractional Laplacian $\Delta^\alpha$ is preserved by the modification in~\eqref{eq:regularized_weight}.

\begin{proposition}\label{pro:decaybound}
	Let $\Delta$ be the Laplacian of an undirected graph $G = (X,\omega,\mathbf{1})$ and $\alpha \in (0,1)$. Then, if $\delta(x_i,x_j) \ge~2$, the following inequality holds
	\begin{equation*}
	\left| \prescript{}{r}{w}^\alpha(x_i,x_j) \right| \le c\left(\frac{\rho(\Delta)}{2\delta(x_i,x_j)-1}\right)^\alpha,
	\end{equation*}
	with a constant $c$ independent from the nodes.
\end{proposition}

\begin{proof}
	The proof follows from Proposition~\ref{prop:fractional_decay_rate}, since by~\eqref{eq:regularized_weight} $\prescript{}{r}{w}^\alpha(x_i,x_j) \equiv {w}^\alpha(x_i,x_j)$ for each $\delta(x_i,x_j) \ge~2$.
\end{proof}

As we have seen from Example~\ref{example:cycleembedding}, the regularization process proposed in~\eqref{eq:regularized_weight} tends to greatly reduce the exploration capabilities with respect to the original non-preserving fractional dynamics. A modification that is useful in recovering this effect is substituting
\begin{equation}\label{eq:regularized_weight_withscale}
\prescript{}{r}{w}^\alpha_\beta(x_i,x_j):= \begin{cases}
-\beta \left(\Delta^\alpha\right)_{i,j} & \mbox{if } \{x_i,x_j\} \notin E,\\
w(x_i,x_j)  & \mbox{if } \{x_i,x_j\} \in E.
\end{cases}, \qquad \beta \geq 1,
\end{equation}
where $\beta$ acts as regularization parameter accounting for the increase in weight we have introduced for the edges in $E$. This is a strategy for artificially filling the gap between the smallest of the weights on $E$ and the largest of the weights on $E'\setminus E$, by increasing the latter. Observe that we want also to maintain the decay behavior highlighted in Proposition~\ref{pro:decaybound}, thus a choice for the $\beta$ parameter satisfying all the modeling requirement is  
\begin{equation}\label{eq:heuristic_for_beta}
    \beta = \frac{\displaystyle \min_{\{x_i,x_j\} \in E} {w(x_i,x_j)}}{\displaystyle \max_{\{x_i,x_j\} \notin E}{w^\alpha(x_j,x_j)}}.
\end{equation}

\section{Numerical examples}\label{sec:numerical_examples}

We use this section to showcase the theoretical properties analyzed in the previous sections on some real-world and model networks. We consider the following networks:
\begin{description}
	\item[\texttt{Karate}] a social network of a university karate club with $n=34$ nodes; it is an undirected weighted graph.
	\item[\texttt{Netscience}] a coauthorship network of scientists         
	working on network theory and experiment. As compiled by M. Newman in May  
	2006, we consider here its largest connected component made by 379 scientists.
	\item[\texttt{ca-GrQc}] a coautorship network from the e-print \texttt{arXiv} that covers papers submitted to the ``General Relativity and Quantum Cosmology'' category. The data covers papers in the period from January 1993 to April 2003.
	\item[\texttt{GD97\_b}] a small weighted undirected graph with $n=46$ nodes on the largest connected component from the ``Pajek network: Graph Drawing contest 1997''.
\end{description}
Each of them can be obtained from the \textit{SuiteSparse Matrix Collection}~\cite{SuiteSparse}. Moreover, we use the Barab\'asi-Albert network model~\cite{Barabasi509} to investigate random scale-free networks obtained by means of a preferential attachment mechanism. These are obtained through a generative model starting from an initial connected graph with $n_0$ nodes. The algorithm then proceeds by adding nodes to the network one at a time and connecting them to $ n \leq n_{0}$ existing nodes, with a probability that is proportional to the number of edges the existing nodes already have.

\subsection{Structural decay of the entries}\label{sec:numexp_decay} We consider here for the \texttt{Netscience} network the bound depicted in Proposition~\ref{pro:decaybound}, and show that the same bound applies both to the standard and the regularized version of the fractional Laplacian.

\begin{figure}[htbp]
	\centering
	\includegraphics[width=0.8\columnwidth]{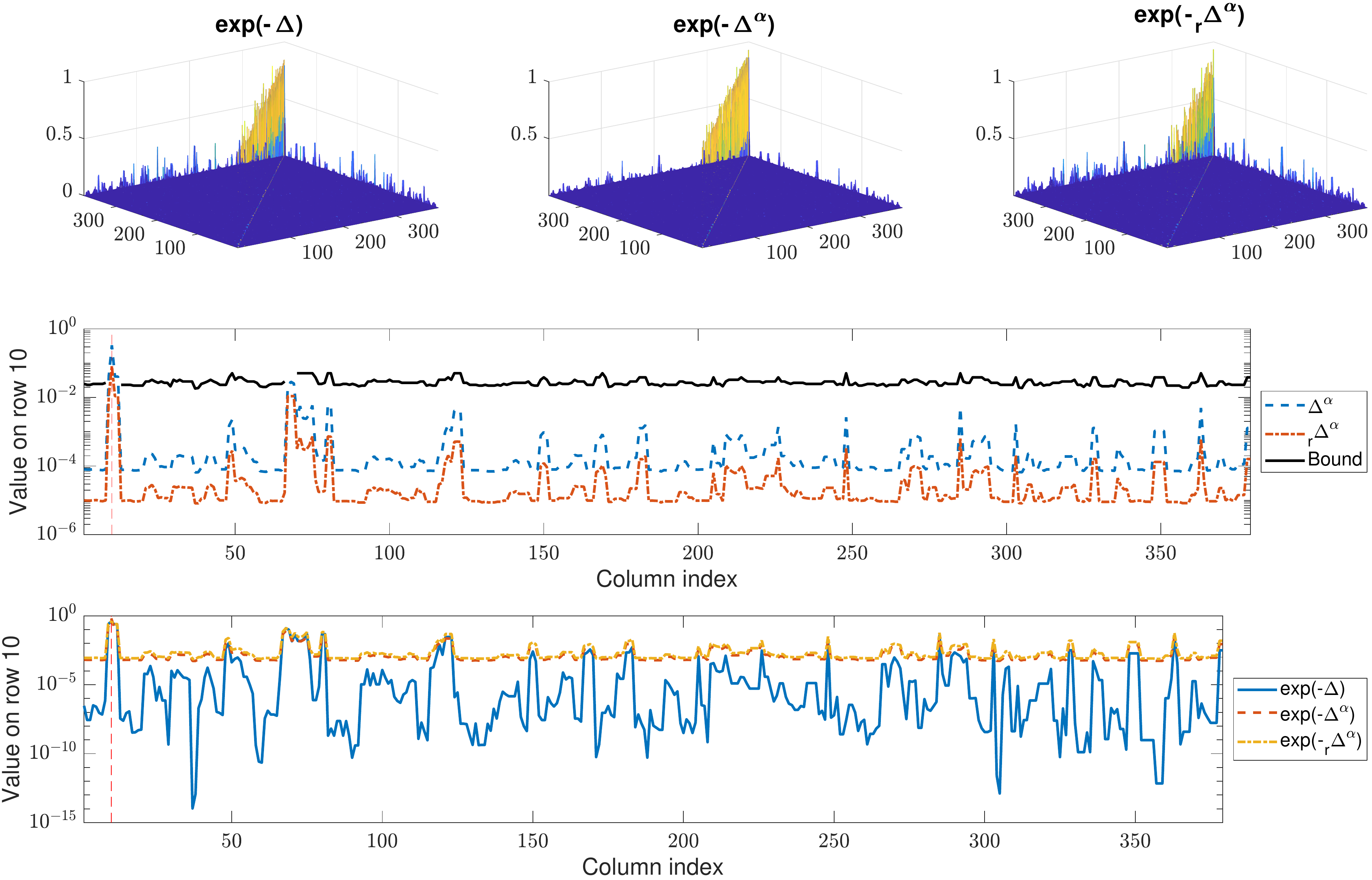}
	\caption{Decay bound for the matrices $\Delta^\alpha$, $\prescript{}{r}{\Delta}^\alpha$ for the complex newtork \texttt{Netscience}, and decay behavior for the relative matrix exponential. All the quantities have been normalized by dividing for the largest entry of the respective matrix, i.e, so that the maximum attained value is always one.}
	\label{fig:decaybound}
\end{figure}

As we observe from Figure~\ref{fig:decaybound}, the reason for which one moves to a non-local version of the random-walk is preserved also by the regularized version, since the decay is mostly unaffected from the regularization process. Indeed, this also shows that we have some margin to set the parameter $\beta$ in~\eqref{eq:regularized_weight_withscale}, to adjust the magnitude of the decay of the fractional weights in the regularized version while maintaining the overall behavior. 

\subsection{Average return probability}\label{sec:numexp_art}
Denoting by $p_{i,j}(t)$ the probability of finding the random walker in the node $j$ at time $t$ starting from the node $i$ at time $t=0$, then we can consider the master equations
\begin{equation*}
\frac{d p_{i,j}(t)}{dt} = - \sum_{l=1}^{N} \hat{\Delta}^\alpha_{l,j} p_{i,l}(t), \qquad \frac{d p_{i,j}(t)}{dt} = - \sum_{l=1}^{N} \prescript{}{r}{\hat{\Delta}}^\alpha_{l,j} p_{i,l}(t),
\end{equation*}
evolving the dynamics of the random walker on the corresponding graphs. To study the diffusive transport we look at the average return probability defined, respectively, by
\begin{equation*}
p_0^{(\alpha)}(t) = \frac{1}{N} \sum_{m=1}^{N} \exp\left(-\lambda_m(\hat{\Delta}^\alpha) t\right), \text{ and } \prescript{}{r}{p_0^{(\alpha)}}(t) = \frac{1}{N} \sum_{m=1}^{N} \exp\left(-\lambda_m(\prescript{}{r}{\hat{\Delta}}^\alpha) t\right).
\end{equation*}
Figure~\ref{fig:average_return_probability} reports the two quantities for the normalized fractional Laplacian $\hat{\Delta}^\alpha$ and for its regularized version $\prescript{}{r}{\hat{\Delta}}^\alpha$. 

\begin{figure}[b]
	\centering
	\subfloat[\texttt{Karate} network, $n = 34$ nodes]{\includegraphics[width=0.8\columnwidth]{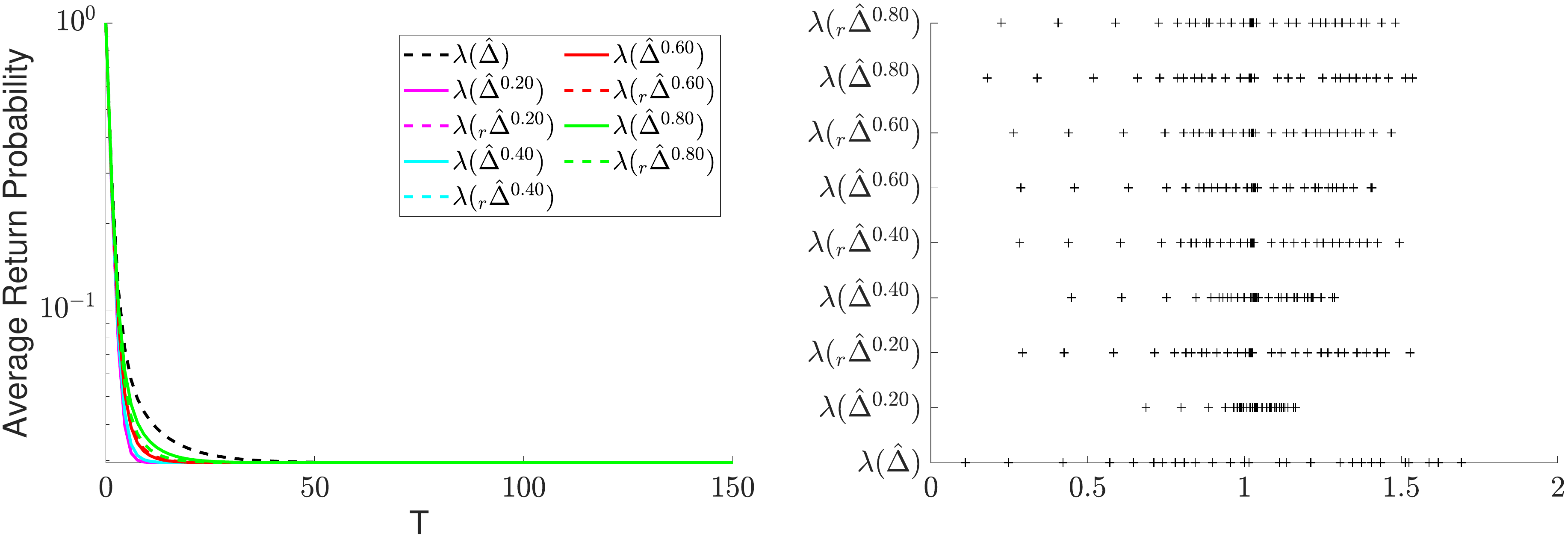}}
	\phantomcaption
\end{figure}
\begin{figure}[b]
\centering
\ContinuedFloat
	\subfloat[\texttt{Netscience} network, $n = 379$]{\includegraphics[width=0.8\columnwidth]{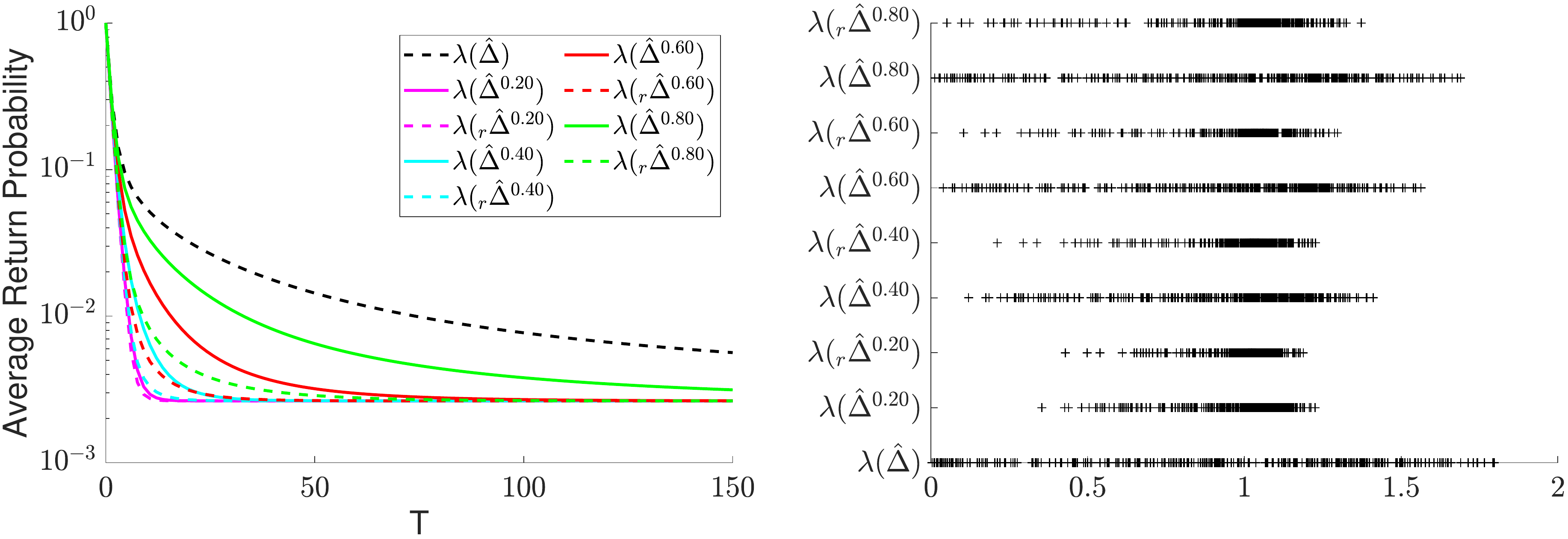}}
	\phantomcaption
\end{figure}
\begin{figure}[t]
\centering
\ContinuedFloat
	\subfloat[\texttt{ca-GrQc} network, $n = 4158$]{\includegraphics[width=0.8\columnwidth]{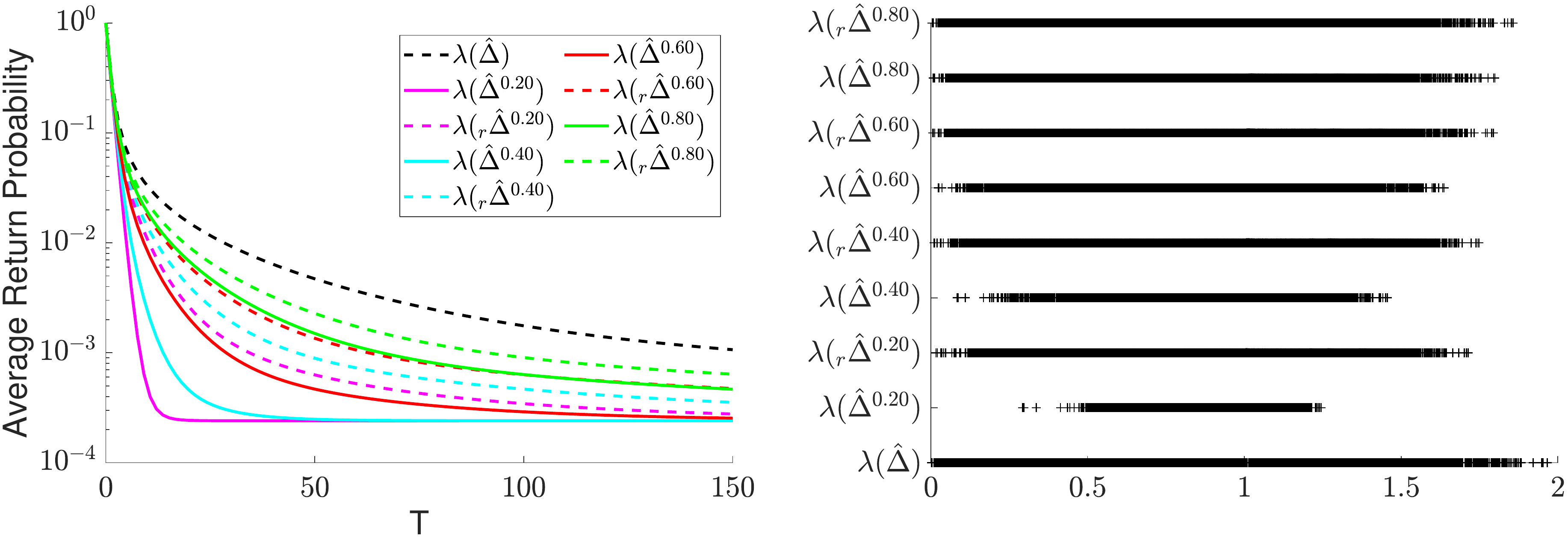}}
	
	\caption{Average return probability. The left panels represent the curves for $p_0^{(\alpha)}(t)$ (continuous lines) and for $\prescript{}{r}{p_0^{(\alpha)}}(t)$ (dashed lines), while the panels on the right represent the complete spectra for $\hat{\Delta}^\alpha$ and $\prescript{}{r}{\hat{\Delta}}^\alpha$.}
	
	\label{fig:average_return_probability}
\end{figure}

What we observe in all the cases is that the regularized Laplacian~$\prescript{}{r}{\hat{\Delta}}^\alpha$ still enhances the exploration of the underlying network, i.e., the $\prescript{}{r}{p_0^{(\alpha)}}(t)$ are always below the one relative to the \emph{local} random walk, while being ``slower'' than the corresponding version induced by~$\hat{\Delta}^\alpha$. Indeed, maintaining the relative weights of the local part of the walks slows down the exploration. This can also be seen by looking at the behavior of the eigenvalues of the respective normalized Laplacian matrices. To improve the latter behavior, we can test the parametric version of the regularization in~\eqref{eq:regularized_weight_withscale} by taking the heuristic for the $\beta$ parameter from~\eqref{eq:heuristic_for_beta}. In Figure~\ref{fig:returnprobability_rebalanced} we compare the average return probability for the different versions.
\begin{figure}[htbp]
    \centering
    \subfloat[\texttt{Karate} network, $n = 34$ nodes]{\includegraphics[width=0.8\columnwidth]{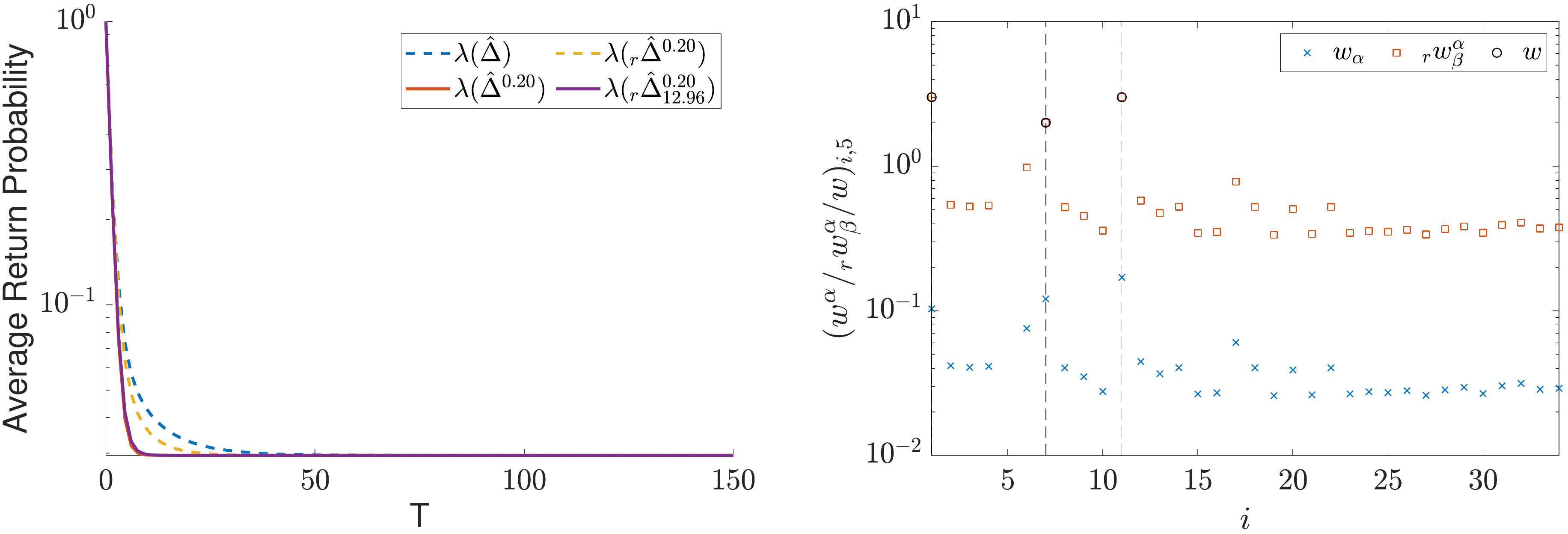}}
    
    \subfloat[\texttt{Netscience} network, $n = 379$]{\includegraphics[width=0.8\columnwidth]{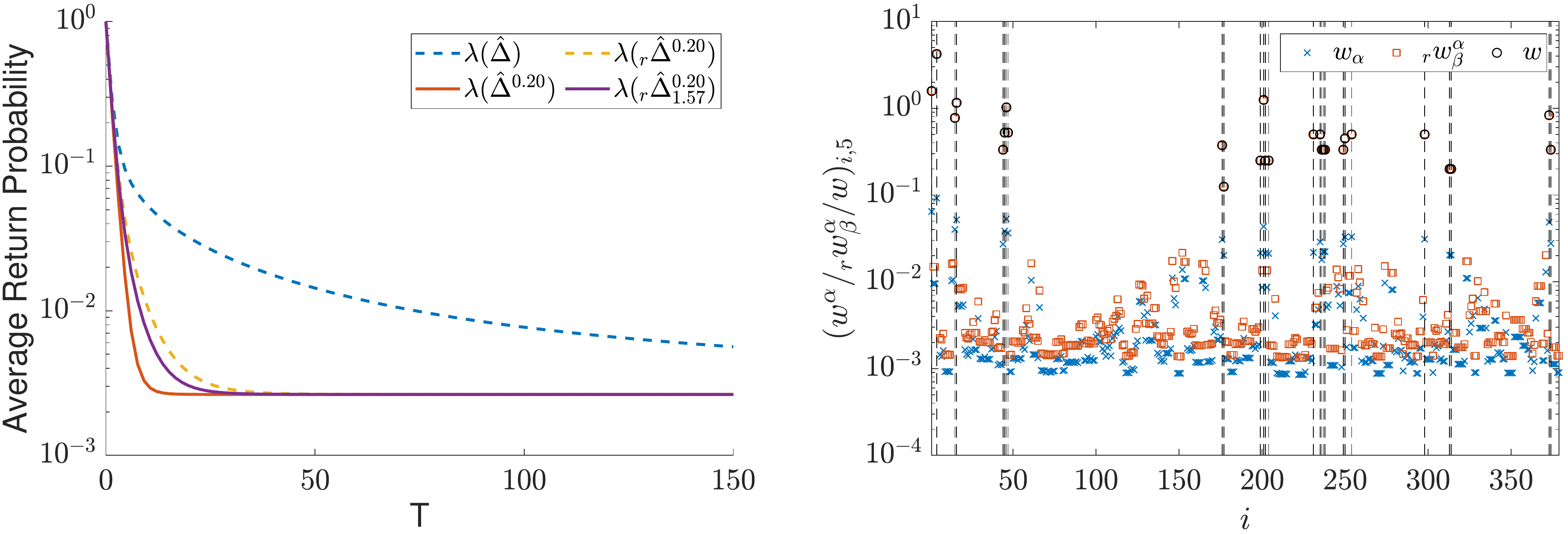}}
    
    \caption{Comparison of the average return probability for the random walks on $G$, $G^\alpha$, $\prescript{}{r}{G}^\alpha$ and $\prescript{}{r}{G}^{\alpha,\beta}$, where the latter is the graph obtained with the parametric version of the regularization in~\eqref{eq:regularized_weight_withscale} for the choice of $\beta$ given in~\eqref{eq:heuristic_for_beta}. The right panels reports the magnitude of the weights $G$, $\prescript{}{r}{G}^\alpha$, and $\prescript{}{r}{G}^{\alpha}_{\beta}$ on a fixed row to showcase the increase in magnitude given by the parameter $\beta$. The vertical dashed lines represent the position of the non-zero weight on the original graph $G$.}
    \label{fig:returnprobability_rebalanced}
\end{figure}
What we observe is that increasing the $\beta$ parameter, for large values of $\beta$ we get better exploration of the underlying network while keeping both the embedding of the original dynamics in the fractional one, and the decay behavior for the probabilities with respect to the length of the jumps. We stress that the choice for $\beta$ in~\eqref{eq:heuristic_for_beta} can be improved while keeping the decay behavior compatible with the result in Proposition~\ref{pro:decaybound}. Finer choices can be possible by analyzing more precisely the values of the weights.

\subsection{Average trapping and mean first passage times}\label{sec:numexp_att} For a graph $G$ with transition probability matrix $P$ we define the \emph{mean first passage time} from node $x_i$ to node $x_j$, denoted by $F_{i,j}$, as the expected time for a walker starting from node
$x_i$ to first reach node $x_j$. The \emph{average trapping time} $\overline{F}_{j}$ for the same random walk is defined as the average of $F_{i,j}$ over all starting vertices $x_i$ to a given trap node $x_j$.

Both quantities can be expressed in terms of the stationary distribution $\pi$ for the row-stochastic matrix $P$ associated with $G$~\cite{ZhangRandomWalks}, indeed 
\begin{equation}\label{eq:stationarydistribution}
    \pi = (\pi_1,\ldots,\pi_n) = \left( \frac{\mu(x_1)}{\mu},\ldots,\frac{\mu(x_n)}{\mu}\right), \quad \mu=\sum_{i=1}^{n} \mu(x_i), 
\end{equation}
therefore
\begin{equation*}
    F_{i,j} = \frac{1}{\pi_j} \sum_{k=2}^{n} \frac{1}{1-\lambda_k} \left(\psi_{k,j}^2 - \psi_{k,i}\psi_{k,j} \sqrt{\frac{\mu(x_j)}{\mu(x_i)}}\right),
\end{equation*}
where $\psi_{j}$ is the $j$th normalized eigenvector of $P$, and
\begin{equation*}
    \overline{F}_{j} = \frac{1}{1 - \pi_j} \sum_{i=1}^{n} \pi_i F_{i,j}.
\end{equation*}

Let us first compare the stationary distributions for transition probability matrices associated with $G^\alpha$ and with $\prescript{}{r}{G}^{\alpha}$, for a weighted version of an undirected Barab\'asi-Albert network~\cite{Barabasi509} in which the weight function $w(x_i,x_j) = (\deg(x_i)\deg(x_j))^{\theta}$, $\theta \in (0,1)$. This is an example of a scale-free network, i.e., of a network whose degree distribution follows a power law. From what we have seen in Proposition~\ref{pro:decaybound} and the expression in~\eqref{eq:stationarydistribution}, we expect that the stationary distribution $\prescript{}{r}{\pi^\alpha}$ of $\prescript{}{r}{G}^{\alpha}$ will be almost indistinguishable from the one on the original graph $G$. The weights we have added to the new edges of the supergraph are decaying with a power-law behavior, and thus are much smaller than the weights on the edges of the original graph. 
\begin{figure}[htbp]
    \centering
    \input{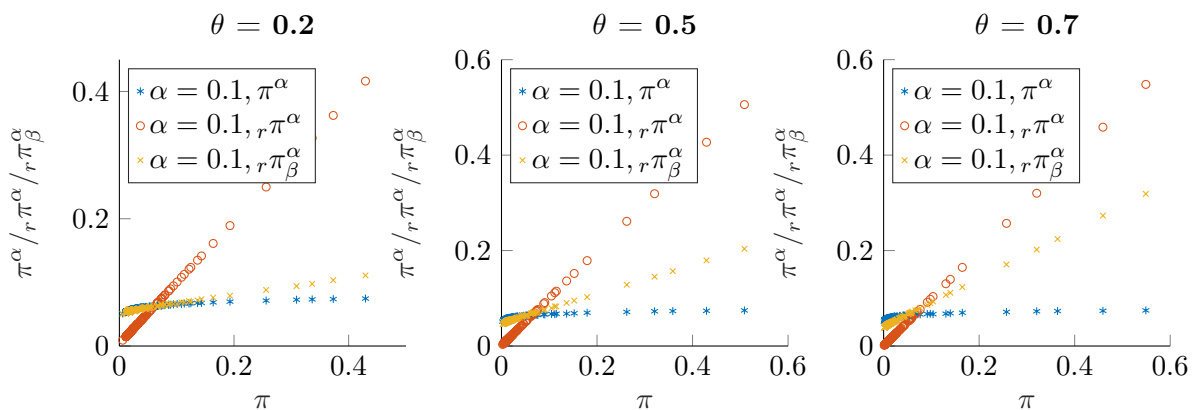}
    \caption{Scatter plot of the stationary distributions of $\prescript{}{r}{G}^{0.1}$, $G^{0.1}$ and $G^{0.1}_{\beta}$ with respect to the distribution $\pi$ of $G$ for a Barab\'asi-Albert scale-free network of $n=300$ nodes with weights $w(x_i,x_j) = (\deg(x_i)\deg(x_j))^{\theta}$.}
    \label{fig:stationarydistribution}
\end{figure}
The scatter-plot in Figure~\ref{fig:stationarydistribution} confirms that the values of $\prescript{}{r}{G}^{\alpha}$ and $\pi$ are aligned, and that the usage of the $\beta$ parameter in~\eqref{eq:regularized_weight_withscale} permits to move in between the two behaviors, i.e., to maintain the embedding/compatibility, and to enhance the exploration that is having smaller trapping times. Requesting the preservation of the original dynamics inside the supergraph without re-scaling the other weights has indeed a strong effect on the stationary distribution. This is then reflected also in the average trapping times that are reported for the same network and $\theta = 0.2$ in Figure~\ref{fig:att_barabasialbert}.
\begin{figure}[htbp]
    \centering
    \input{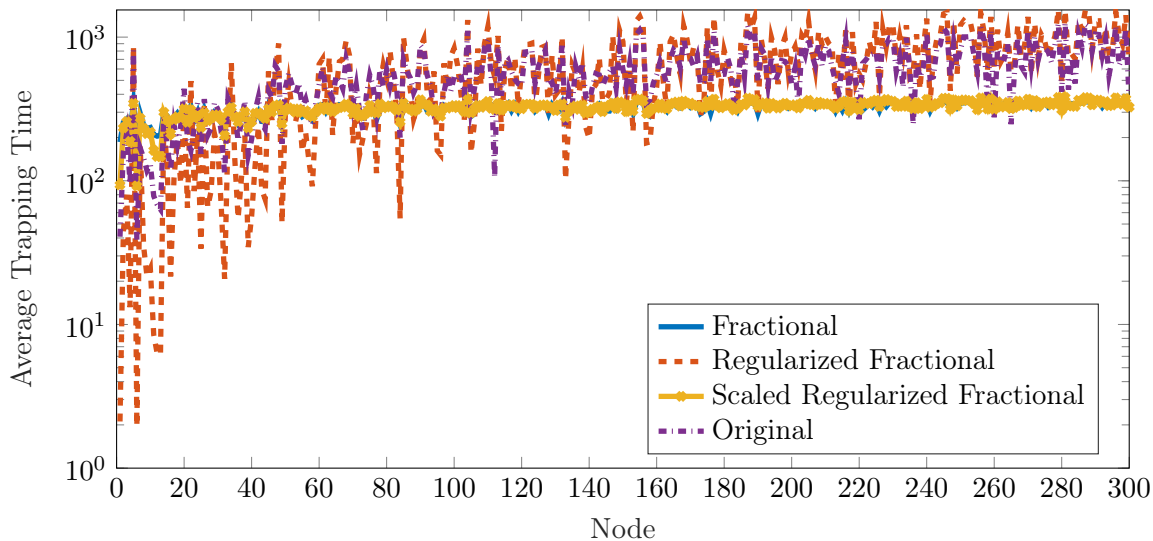}
    \caption{Average trapping time for all the nodes of the networks $G$, $\prescript{}{r}{G}^{0.1}$, $G^{0.1}$ and $G^{0.1}_{\beta}$ with respect to the distribution $\pi$ of $G$ for a Barab\'asi-Albert scale-free network of $n=300$ nodes with weights $w(x_i,x_j) = (\deg(x_i)\deg(x_j))^{0.2}$.}
    \label{fig:att_barabasialbert}
\end{figure}
This behavior suggests that, whenever we enforce an exact conversation of the bias given by the original weights of the network, the asymptotic measure will be pretty much indistinguishable from the one of the original network. This is indeed also confirmed when we apply the same procedure to the \texttt{Karate} network, as depicted in Figure~\ref{fig:average_trapping_time_karate}. 
\begin{figure}[htbp]
    \centering
    \includegraphics[width=0.8\columnwidth]{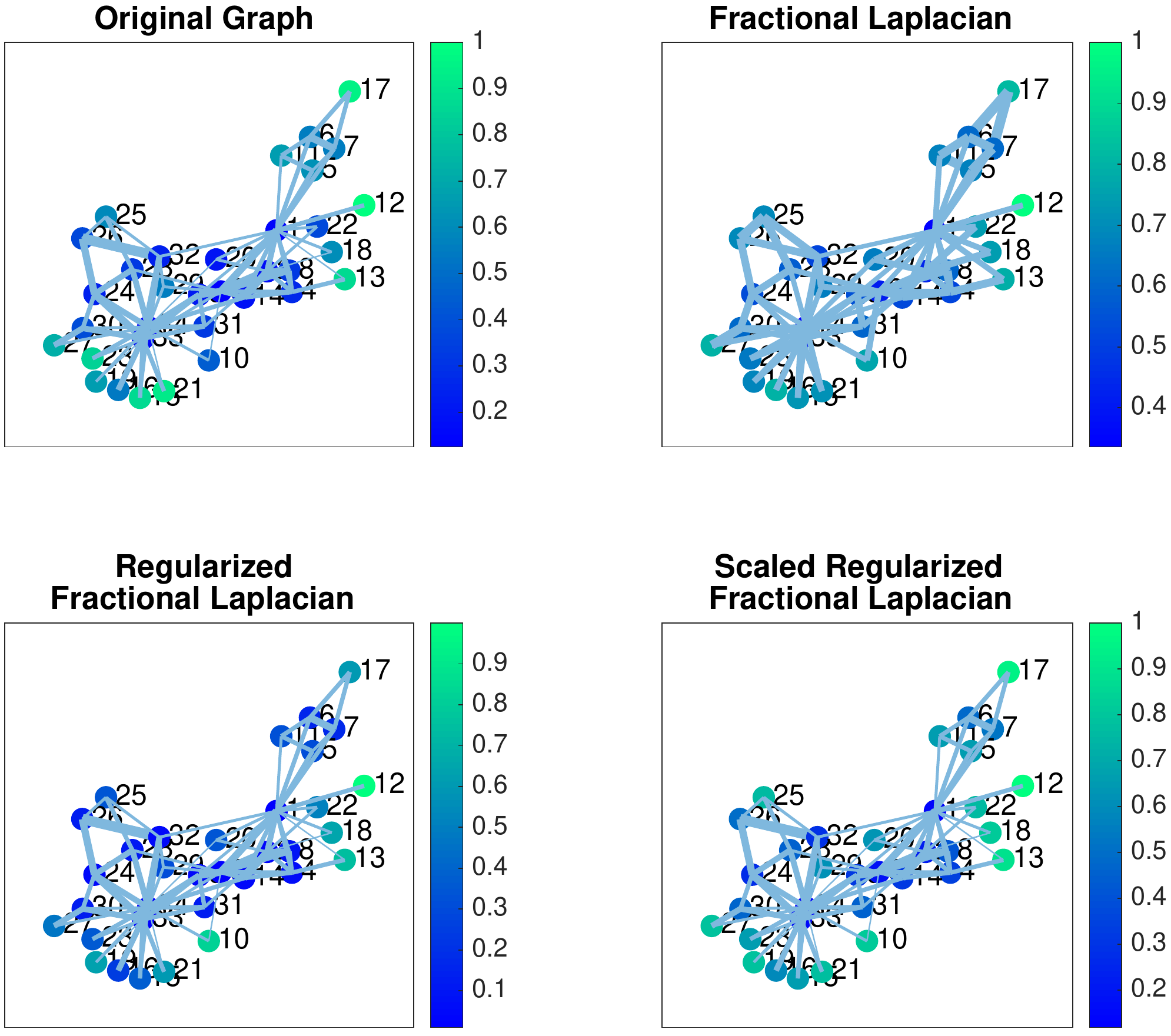}
    \caption{Average trapping time for all the nodes of the \texttt{Karate} network. The up-left panel is relative to the dynamics induced on $\hat{G}$, the up-right one is for the dynamics induced on $G^{0.2}$, the down-right one for the dynamics induced on $\prescript{}{r}{G}^{0.2}$, and the last one for $G^{0.2}_{16.85}$. In all the cases we have depicted only the edges of the original graph $G$, the width of the edges is proportional to their weight, while the color of the nodes represent the Average Trapping Time. To better confront them all the times have been normalized to be in $[0,1]$.}
    \label{fig:average_trapping_time_karate}
\end{figure}
On the one hand, we observe that the regularized fractional Laplacian from \eqref{eq:regularized_weight} exactly preserves the weights on the edge of the original graph $G$ (compare the width of the depicted edges), on the other hand the average trapping time between the original graph and the one induced by the regularized fractional Laplacian are very similar and we can use again the parameter $\beta$ in~\eqref{eq:regularized_weight_withscale} to adjust the overall behavior. 

\subsection{Stochastic equivalence for $G_\alpha$} We briefly consider here a numerical example for the stochastic equivalence for the case of the path graph Laplacian. This is an example on a graph with a more complex topology than the one in Counterexample~\ref{cexample:dpathlaplacian}. In Figure~\ref{fig:kpathexample} we consider the case of the Mellin transformed path Laplacian combined with the shortest-path distance, for $\alpha = 2$. We build it from the description in~\eqref{eq:path-graph-operator} by selecting $h_2(t) = t^{-2}$, where $d_w(x_i,x_j)$ is the usual extension of the shortest path distance $\delta$ for a graph $G$ with weight function $w:X\times X\rightarrow [0,+\infty)$, i.e., the distance $d_w(x_i,x_j)$ is the value of the minimal sum $\displaystyle \sum w(x_p,x_q)$ over all possible paths, i.e.,
\begin{equation*}
    d_w(x_i,x_j) = \arg \min_{x_i \sim \ldots \sim x_j} \left\lbrace \sum_{x_p \sim x_q} w(x_p,x_q) \,:\, x_i \sim \ldots \sim x_p \sim x_q \sim \ldots \sim x_j \right\rbrace.
\end{equation*}
We readily observe that the weights on the modified graph $G_2$ relative to the edges that were originally in $G$ are distributed differently from the ones on $G$.
\begin{figure}[htbp]
    \centering
    \subfloat[Depiction of the graphs $G$ (left panel) and $G_2$ (right panel). The edges on the original $G$ have a width that is proportional to their weight. The added edges in $G_2$ are reported as dashed lines.\label{fig:kpathexample_a}]{\includegraphics[width=0.8\columnwidth]{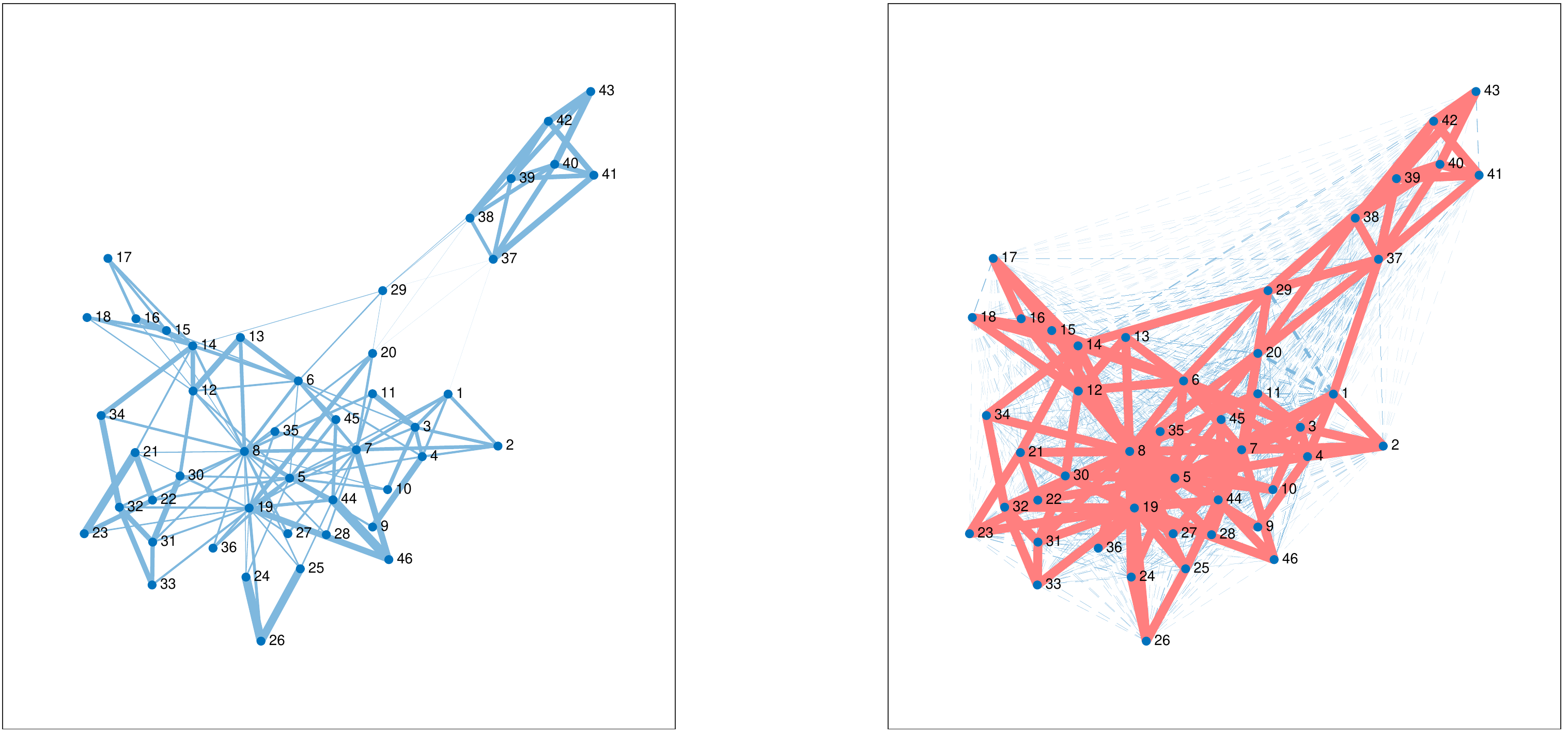}}
    
    \subfloat[Compatibility condition for the weights of $G$ and $G_\alpha$ for $G$ the \texttt{GD97\_b} graph. The edges are numbered in lexicographical order with respect to the number associated to the vertex.\label{fig:kpathexample_b}]{% This file was created by matlab2tikz.
\definecolor{mycolor1}{rgb}{0.00000,0.44700,0.74100}%
\begin{tikzpicture}

\begin{axis}[%
width=4.4in,
height=2in,
at={(0in,0in)},
scale only axis,
xmin=1,
xmax=132,
xticklabel style={rotate=80,font=\tiny},
xlabel style={font=\color{white!15!black}},
xlabel={Edge},
ymode=log,
ymin=0.1,
ymax=1000,
yminorticks=true,
ylabel style={font=\color{white!15!black}, align=center},
ylabel={$k_{\alpha,1}(x_i,x_j)/k_{\alpha,1}(x_i,x_k)$\\[1ex]$\forall x_i,x_j,x_k : x_i \sim x_j, x_i \sim x_k$},
axis background/.style={fill=white},
legend style={legend cell align=left, align=left, draw=white!15!black}
]
\addplot [color=mycolor1, mark=o, mark options={solid, mycolor1}]
  table[row sep=crcr]{%
1	4837.0422362595\\
2	7386.56860791678\\
3	2911.85184276841\\
4	3667.65308615847\\
5	217.973294643998\\
6	0.000401361626196457\\
7	2.61145225842\\
8	0.963699468939146\\
9	30.3926922734473\\
10	6.70816746683159\\
11	6.84954256719499\\
12	5.42213141484019\\
13	1.96779139543055\\
14	9.6667271394612\\
15	3.74373923274394\\
16	7.2550035133222\\
17	8.80058431164345\\
18	7.3387810092545\\
19	6.74264162772393\\
20	1.62868655296557\\
21	1.74842194182767\\
22	10.6278479803996\\
23	4.24254560580289\\
24	3.30257669933987\\
25	1.04351996899466\\
26	2.03198434587772\\
27	1.84364614876032\\
28	3.70483825753089\\
29	21.6939014178936\\
30	8.11967657066209\\
31	6.62576084901992\\
32	20.3578001432411\\
33	10.3438190570681\\
34	97.4227924387586\\
35	7.09567968444854\\
36	3.38413146276108\\
37	302.704155463923\\
38	18.1973314333806\\
39	72.726092971642\\
40	52.2467330998245\\
41	58.8854120140681\\
42	181.142955859462\\
43	15.7905777339905\\
44	27.4259992779276\\
45	17.2745984275192\\
46	26.6767569244685\\
47	40.3862577659462\\
48	1062.00852974169\\
49	0.980372485303606\\
50	0.480465671128772\\
51	9.57249939911569\\
52	3.02173758524209\\
53	26.0271999512206\\
54	41.3670983953615\\
55	0.726032392813202\\
56	0.768393982835303\\
57	3.10185950976238\\
58	6839.39005740915\\
59	1.36988715436983\\
60	2.5100726803785\\
61	0.34000900863634\\
62	3.0821397704138\\
63	6.4346447926317\\
64	2.07550077768835\\
65	6.72886368052921\\
66	11.980829240339\\
67	9.60813636404701\\
68	5.21753849391983\\
69	2.86111283379154\\
70	26.6711208001815\\
71	35.8517143188301\\
72	29.7963033918483\\
73	0.177695523688761\\
74	7.78869586215795\\
75	23.3983378842573\\
76	24.5909671030486\\
77	10.4125416892865\\
78	9.59749635127796\\
79	13.4031936136923\\
80	46.4614090306188\\
81	155.723717461294\\
82	3.12597013864091\\
83	0.476597819398709\\
84	1.07244729794956\\
85	0.412960626717472\\
86	26978.7320553946\\
87	0.952850126440701\\
88	0.687867647848107\\
89	0.233524232316086\\
90	1.34580136971969\\
91	7.42721043628108\\
92	1.64950274766338\\
93	1.87005596017833\\
94	1.72260167598371\\
95	3.76634691385538\\
96	6.09320108078126\\
97	3.25685481027585\\
98	2.65142192124528\\
99	29.0561957212657\\
100	119.76087170888\\
101	0.000314625041990273\\
102	0.148027652688415\\
103	10.7113396902449\\
104	13.1787029967818\\
105	1.55226164977554\\
106	16.048599964932\\
107	16.0492770698593\\
108	0.00176010825386157\\
109	0.828111747394502\\
110	2.16763725507964\\
111	8.1649253438184\\
112	7.6482797592303\\
113	2.15488421329207\\
114	3.98522528747426\\
115	8.34210464754045\\
116	12.0818739370467\\
117	24.7647032439798\\
118	174.225868081708\\
119	287.53801933253\\
120	2.19743801283222\\
121	5.78346669526406\\
122	18.9915616051841\\
123	5.8215104052034\\
124	7.18141085853088\\
125	8.0652084200183\\
126	6.10156447136986\\
127	2.3349614286619\\
128	4.03487879065255\\
129	2.35248907119952\\
130	3.07181034600703\\
131	3.53656034636706\\
132	78.9557077956691\\
};

\addplot [color=black, dashed, forget plot]
  table[row sep=crcr]{%
1	1\\
132	1\\
};
\end{axis}
\end{tikzpicture}%
}
    \caption{Counterexample for the embedding of $\hat{G}$ in $\hat{G}_2$, when the Mellin transformed Laplacian is used combined with the shortest-path distance, and $G$ is the \texttt{GD97\_b} graph. The first two figures in \eqref{fig:kpathexample_a} depict the two graphs and highlight the weights of the edges by printing them with a proportional width; the plot \eqref{fig:kpathexample_b} depicts the compatibility condition in Lemma~\ref{lem:embedding_path-graph}.}
    \label{fig:kpathexample}
\end{figure}
Indeed, this is confirmed when we look explicitly at Figure~\ref{fig:kpathexample_b}, where we check condition \eqref{eq:001} in Lemma~\ref{lem:embedding_path-graph}. The ratio between the weights of the corresponding graphs is far from being one, and thus the embedding of $\hat{G}$ in $\hat{G}_2$ is impossible. This leads to the incompatibility of the dynamics. Similar results can be obtained for other values of $\alpha$, using the Laplace transformed Laplacian, or changing the distance defining the $k_{\alpha,n}$ as described in Section~\ref{ssection:path_Laplacian}. While path-Laplacians preserve the original dynamics for the combination of unweighted graphs and shortest path distance, this property is in general lost when moving to weighted settings or making use of other distances.

\section{Conclusions}\label{sec:conclusions}
In this work we have shown that generating non-local dynamics from the graph Laplacian of a starting (finite) graph can have some drawbacks. If we still care about the original model then some regularization is mandatory, otherwise in general the new dynamics will cease to be compatible with the dynamics of the model graph itself. This awareness leads to new interesting questions for future investigations: what does happen in the infinite graph case? As a matter of fact, the graph Laplacian can approximate second-order linear differential operators, see \cite{Burago2014} and \cite{Adriani2020} for some applications. Therefore, can we infer the same conclusions about compatibility and embedding for the continuous Laplacian? This last question has an application purpose as well since, even in the continuous setting, many models that implement anomalous diffusion are based on the (spectral) fractional Laplacian, see e.g. \cite{Metzler2000,Donatelli2016}. 

\section*{Acknowledgement}
The first author is member of GNAMPA-INdAM. The work of all authors has been partially supported by the GNCS-INdAM.

\printbibliography	
\end{document}